\documentclass[a4paper,12pt]{amsart}
\usepackage{amssymb,latexsym,amsxtra,amscd,amsfonts}
\usepackage{enumerate}
\usepackage[mathscr]{eucal}
\usepackage{dsfont}
\usepackage{calrsfs}

\numberwithin{equation}{section}

\newtheorem{thm}{Theorem}[section]
\newtheorem{cor}[thm]{Corollary}
\newtheorem{lem}[thm]{Lemma}
\newtheorem{prop}[thm]{Proposition}

\theoremstyle{definition}

\newtheorem{rem}[thm]{Remark}
\newtheorem{exmp}[thm]{Example}

\newcommand{\thmref}[1]{Theorem~\ref{#1}}
\newcommand{\propref}[1]{Proposition~\ref{#1}}
\newcommand{\lemref}[1]{Lemma~\ref{#1}}
\newcommand{\eref}[1]{(\ref{#1})}

\newcommand{\set}[2]{\ensuremath{\left\{ #1\,:\; #2\right\}}}

\newcommand{\N}{\mathds N}
\newcommand{\R}{\mathds R}

\newcommand{\saM}{M_{sa}}
\newcommand{\sstar}{s^\ast(M,M_\ast)}
\newcommand{\centre}[1]{\mathcal Z(#1)}
\newcommand{\vNa}{von Neumann algebra}

\newcommand{\Ca}{$C^\ast$-algebra}

\newcommand{\st}[1]{\xrightarrow{\,s({#1},{#1}_\ast)\,}}

\newcommand{\sorder}[1]{\xrightarrow{\,\tau_{os}(#1)\,}}

\newcommand{\ro}{\ensuremath{\varrho}}

\begin{document}
\title{The order topology for a von Neumann algebra}
\author{Emmanuel Chetcuti\textsuperscript{1}, Jan Hamhalter\textsuperscript{2} and Hans Weber\textsuperscript{3}}

\begin{abstract}
The order topology $\tau_o(P)$ (resp. the sequential order topology $\tau_{os}(P)$) on a poset $P$ is the topology that has as its closed sets those that contain the order limits of all their order convergent nets (resp. sequences).  For a von Neumann algebra $M$ we consider the following three posets: the self-adjoint part $\saM$, the self-adjoint part of the unit ball $\saM^1$, and the projection lattice $P(M)$.  We study the order topology (and the corresponding sequential variant) on these posets, compare the order topology  to the other standard locally convex topologies on $M$, and relate the properties of the order topology to the underlying operator-algebraic structure of $M$.
\end{abstract}
\date{}
\maketitle

\begin{center}  

{\small\textsuperscript{1} Department of Mathematics, Faculty of Science, University of Malta,  Msida MSD.06, Malta\\
E-mail: {\tt emanuel.chetcuti@um.edu.mt}\\
\textsuperscript{2} Czech Technical University in Prague, Faculty of Electrical Engineering, Department of Mathematics, Technick\'a 2\\ 166 27
Prague 6,  Czech Republic\\
E-mail: {\tt hamhalte@math.feld.cvut.cz}\\
\textsuperscript{3}Dipartimento di matematica e informatica,
Universit\'a degli Studi di Udine,
1-33100 Udine,
Italia\\
E-mail: {\tt hans.weber@uniud.it}}

\end{center}

\section{Introduction}

Order convergence has been studied in the context of posets and lattices by various authors  \cite{Birkhoff1,Birkhoff2, Kantorovich}  (see also \cite{Frink,McShane,MaAn}).  The  order topology on a poset is defined to be  the finest topology preserving order convergence.

In  \cite{Palko,BCW} the order topology for the lattice of projections  acting on a Hilbert space was studied.  It is the aim of the present paper to give a first systematic treatment of various order topologies associated with  a von Neumann algebra. We show that the  properties of these topologies are nicely connected with the inner structure of the underlying algebra and with the locally convex topologies living on it.

We first consider the self-adjoint part $\saM$ of a von Neumann algebra $M$ and study the order topology $\tau_o(M_{sa})$ induced by the standard operator order.  We prove that when $M$ is $\sigma$-finite,  sequential convergence
w.r.t. $\tau_o(M_{sa})$ coincides with  sequential convergence w.r.t. the $\sigma$-strong topology.
The proof is based on  noncommutative Egoroff's Theorem.
As a consequence, one obtains that on bounded parts of $M_{sa}$  the order topology coincides with any of the locally convex topologies on $M$ that is compatible with the duality $\langle M,M_\ast\rangle$ where $M_\ast$ is the unique predual of $M$.  Our result is sharp in the sense that
the $\sigma$-strong topology coincides with the order topology $\tau_o(\saM)$  if and only if $M$ is finite-dimensional.
The fact that the order topology on ordered vector spaces is in general far from being a linear topology makes this coincidence rather surprising. Another interesting feature of this result is the possibility to recover
(on bounded parts)  the locally convex topologies arising from the duality $\langle M,M_\ast\rangle $ (a component of the von Neumann structure) only from the order (a component of the $C^\ast$-structure).  More precisely, we are saying that if $M$ and $N$ are $\sigma$-finite von Neumann algebras such that $\saM$ and $N_{sa}$ are order-isomorphic (i.e. there exists a bijection preserving the order in both directions), then the unit balls $M^1$ and $N^1$ are homeomorphic w.r.t. the $\sigma$-strong topologies. 

We then compare the Mackey topology $\tau(M,M_\ast)$ with the order topology $\tau_o(M_{sa})$ and the sequential variant $\tau_{os}(M_{sa})$. 
The Mackey topology is coarser than  $\tau_{os}(M_{sa})$ and we characterize von Neumann algebras for which $\tau(M,M_\ast)=\tau_{os}(\saM)$.  Indeed, we prove that this happens if and only if $M$ is $\ast$-isomorphic to a countable direct sum of finite-dimensional full matrix algebras.  From a topological point of view this happens exactly when any of the following conditions is satisfied: {\rm(i)}~$M$ is $\sigma$-finite and $M^1$ is compact w.r.t.  the $\sigma$-strong$^\ast$ topology, {\rm(ii)}~the Mackey topology is sequential, {\rm(iii)}~$M$ is $\sigma$-finite and $\tau_o(\saM)$ is a linear topology.  The proof of these results rest heavily on the technique of mixed topologies. That is why we study mixed topologies and develop results there that we believe can be of independent interest.  Using \cite{Akemann} we  show that the Mackey topology is equal to the mixed topology of the norm topology and the $\sigma$-strong$^\ast$ topology.  This  is in fact a noncommutative extension of the interesting result of M.~Novak \cite{Novak} saying that  the Mackey topology on $L^\infty$ coincides with the mixed topology of the norm topology and the topology of convergence in measure.  Although not investigated here, we believe that this equality can contribute to the problem studied by J.F.~Aarnes \cite{Aarnes} of whether the Mackey topology of a von Neumann subalgebra coincides with the restriction  of the Mackey topology of the big algebra.

In the last section we consider as posets the  projection lattice $P(M)$ and the self-adjoint part of the unit ball $M_{sa}^1$.  Unless the algebra is abelian, the order topology on neither of these posets  coincides with the restriction of the global order topology $\tau_o(\saM)$.  In fact, we show that if $M$ is $\sigma$-finite then the following conditions are equivalent: {\rm(i)}~$M$ is of finite Type, {\rm(ii)}~the order topology on $M_{sa}^1$ and the $\sigma$-strong operator topology restricted to $M_{sa}^1$ have the same  null sequences {\rm(iii)}~the order topology on the projection lattice $P(M)$ and the $\sigma$-strong operator topology restricted to $P(M)$ have the same  null sequences. This gives a new characterization of finite von Neumann algebras. 
   
The paper is organized as follows. Section~2 collects basic facts on the order topology on posets and  ordered vector spaces needed later. In  Section~3 results on mixed topologies are isolated. Section~4 deals with the relationship between the standard locally convex topologies and the order topology on $\saM$. Section~5 deals with the order topologies of the projection lattice and the unit ball of a von Neumann algebra.

\section{Preliminary results}
\subsection{The order topology and sequential order topology}\label{s2.1}
Let $(P,\le)$ be a partially ordered set. A net  $(x_\gamma)_{\gamma\in \Gamma}$ is said to \emph{order converge} to $x$ in $(P,\le)$
 (in symbols $x_\gamma\xrightarrow{o}x$) if there exist nets $(y_\gamma)_{\gamma\in \Gamma}$, $(z_\gamma)_{\gamma\in \Gamma}$
 in $P$ such that  $y_\gamma \le x_\gamma\le z_\gamma$ for all $\gamma\in \Gamma$,
 $y_\gamma\uparrow  x$ and $z_\gamma\downarrow x$; i.e.  $(y_\gamma)$ is increasing, $(z_\gamma)$ is decreasing and\footnote{For a subset  $X=\set{x_\gamma}{\gamma\in\Gamma}$ of $P$ we shall denote
 by $\bigvee_{\gamma\in\Gamma}x_\gamma$ and $\bigwedge_{\gamma\in\Gamma}x_\gamma$ the least upper bound and the greatest lower bound of the set $X$, respectively.} $\bigvee_{\gamma\in\Gamma} y_\gamma=x=\bigwedge_{\gamma\in\Gamma} z_\gamma$.

It is easy to see that the order limit of an order convergent
net is uniquely determined. A subset $X$ of $P$ is called \emph{order closed} (resp. \emph{sequentially order closed})  if no net (resp. sequence) in $X$ order converges to a point outside of $X$. The collection of all 
order closed sets (resp.  sequentially order closed sets) comprises the closed sets for some topology, the \emph{order
topology} $\tau_o(P)$ (resp. the \emph{sequential order
topology} $\tau_{os}(P)$) of $P$.  The order topology of $P$ is the finest topology on $P$
that preserves order convergence of nets; i.e. if $\tau$ is a topology on $P$ such that $x_\gamma\xrightarrow{o} x$ in $P$  implies $x_\gamma\xrightarrow{\tau} x$, then $\tau\subseteq \tau_o(P)$. The sequential order topology of $P$ is the finest topology on $P$ that preserves order convergence of sequences. 
 Clearly, $\tau_o(P)\subseteq \tau_{os}(P)$ and we recall that both topologies satisfy $T_1$ but in general are not Hausdorff \cite{Floyd,FloydKlee}.

Although convergence w.r.t. $\tau_o(P)$  does not necessarily imply order convergence, for a sequence converging w.r.t.  $\tau_{os}(P)$ we have the following  useful observation (well-known in a less general setting).

\begin{prop}\label{1.1}\cite[Proposition 2]{BCW} Let $( P,\leq)$ be a partially ordered set, $x\in P$ and $(x_n)_{n\in\N}$
a sequence in $P$.  Then $(x_n)_{n\in\N}$ converges to $x$ w.r.t.  $\tau_{os}(P)$  if and only if any subsequence of $(x_n)_{n\in\N}$ has a subsequence order converging  to $x$.
\end{prop}

The sequential order topology is in general strictly finer than  the order topology, however it turns out that the two topologies coincide when  $P$ is \emph{monotone order separable}.    We call $(P,\le)$  monotone order separable  if for
   every increasing (or  decreasing) net $(x_\gamma)_{\gamma\in\Gamma}$ in $P$ that has a supremum (resp.  infimum) in $P$ there exists an increasing sequence $(\gamma_n)_{n\in\N}$ in $\Gamma$ such that $\bigvee_{n\in\N} x_{\gamma_n}=\bigvee_{\gamma\in\Gamma} x_\gamma$ (resp.  $\bigwedge_{n\in\N} x_{\gamma_n}=\bigwedge_{\gamma\in\Gamma} x_\gamma$).

 \begin{prop}\label{1.2}\cite[Proposition 3]{BCW} Let $(P,\leq)$ be a partially ordered set.  Then $\tau_{os}(P)=\tau_o(P)$ if and only if
$(P,\le)$ is  monotone order separable.
 \end{prop}

Every order convergent sequence is order bounded, and every order convergent net is eventually order bounded. Therefore in the definition of order closed sets it is enough to consider order bounded nets.  We recall that $(P,\le)$ is Dedekind complete if every subset having an upper bound (or a lower bound) has a supremum (resp. an infimum). $(P,\le)$ is conditional monotone complete if every monotonic increasing net (or monotonic decreasing net) having an upper bound (resp. a lower bound) has a supremum (resp. an infimum).  Dedekind $\sigma$-completeness (resp. conditional monotone $\sigma$-completeness) is defined analogously requiring the condition to hold for countable subsets  (resp. sequences).
It is easily seen that when $(P,\le)$ is Dedekind complete, an order bounded net $(x_\gamma)_{\gamma\in\Gamma}$ order converges to $x$ in $(P,\le)$ if and only if $\limsup_\gamma x_\gamma=\liminf_\gamma x_\gamma=x$. When $(P,\le)$ is only assumed to be Dedekind $\sigma$-complete a similar assertion holds for sequences.

If $P_0$ is a subset of $P$ it can very well happen that $\tau_o(P_0)$ and $\tau(P)|P_0$ are incomparable.   However,  we have the following easily seen observations which we put as a proposition for better reference.

\begin{prop}\label{1.6} Let $(P,\leq)$ be a partially ordered set and let $P_0$ be a subset of $P$.
\begin{enumerate}[{\rm(i)}]
\item If $(P,\le)$ is conditional monotone complete and $P_0$ is $\tau_o(P)$-closed then $\tau_o(P)|P_0\subseteq \tau_o(P_0)$.  
\item If $(P,\le)$ is Dedekind complete and $P_0$ is a $\tau_o(P)$-closed sublattice
of $P$ then $\tau_o(P)|P_0= \tau_o(P_0)$.
\end{enumerate}
\end{prop}
An analogous proposition holds for the sequential order topology: Proposition \ref{1.6} remains true if one replaces the order topology by the sequential order topology, conditional monotone completeness  by monotone $\sigma$-completeness and Dedekind completeness by  Dedekind $\sigma$-completeness.

We shall now consider the case when the underlying poset carries also a linear structure.  Let $X$ be an  ordered vector space with positive cone $X^+=\set{x\in X}{x\ge 0}$.  For basic results and terminology
on ordered vector spaces the reader may wish to consult \cite{AliprantisBurkinshaw,LuxemburgZaanen,Schaeffer}.
It is clear that the order topology $\tau_o(X)$ and the sequential order topology $\tau_{os}(X)$ are translation invariant and homogeneous, i.e. if $A$ is a subset of $X$ closed w.r.t. $\tau_o(X)$ (or  $\tau_{os}(X)$) then
$A+x$ and $\lambda A$ are closed w.r.t. $\tau_o(X)$ (resp.  $\tau_{os}(X)$) for every $x\in X$ and $\lambda\in\R$.  In general these topologies however fail to be linear topologies as  the following example shows: Let $\mathfrak{A}$ be the complete Boolean algebra of all regular open subsets of $[0,1]$ and $B(\mathfrak{A})$ be the closed linear span of the set of characteristic functions $\chi_A$, $A\in\mathfrak{A}$, in the space $\bigl(B[0,1],\|\cdot\|_\infty\bigr)$ of all bounded real functions on $[0,1]$ w.r.t. the supremum norm $\|\cdot\|_\infty$. Then $B(\mathfrak{A})$ is a monotone order separable, Dedekind complete Riesz space.  Thus $\tau_{os}(B(\mathfrak{A}))=\tau_{o}(B(\mathfrak{A}))$.  Since $\tau_{o}(B(\mathfrak{A}))$ satisfies $T_1$ and $\tau_o(B(\mathfrak A))$ is not Hausdorff, it follows that $\tau_{o}(B(\mathfrak{A}))$ is not a group topology.\footnote{In \cite{Floyd} it is shown that $\tau_o(\mathfrak A)$ is not Hausdorff.} 

\begin{prop}\label{1.3}
Let $X$ be an ordered vector space and $(a_n)_{n\in\N}$, $(b_n)_{n\in\N}$ sequences in $X$  such that  $a_n \sorder{X} a$, $b_n \sorder{X} b$ and $(\lambda_n)_{n\in\N}$ a sequence in $\R$ such that $\lambda_n\longrightarrow \lambda$. Then:
\begin{enumerate}[{\rm(i)}]
\item $a_n+b_n \sorder{X}  a+b$;
\item If $\bigl(\frac{1}{n}a\bigr)_{n\in \N}$ order converges to $0$ (in particular if $X$ is an Archimedian Riesz space), then $\lambda_n a_n\sorder{X} \lambda a$.
\end{enumerate}
\end{prop}
\begin{proof}
We will apply Proposition \ref{1.1}. Passing to suitable subsequences we may assume that $(a_n)_{n\in\N}$ and $(b_n)_{n\in\N}$ order converge to $a$ and $b$, respectively, and moreover that $|\lambda-\lambda_n|\leq\frac{1}{n}$ and either $\lambda_n-\lambda\ge 0$ for each $n$ or $\lambda_n-\lambda<0$ for each $n$.  Let $(x_n)_{n\in\N}$, $(y_n)_{n\in\N}$, $(u_n)_{n\in\N}$ and $(v_n)_{n\in\N}$ be sequences in $X$ such that
\[
x_{n}\le a_{n}\le y_{n}\qquad u_{n}\le b_{n}\le v_{n}\qquad\text{for all }n\in\N\,,
\]
and $x_{n}\uparrow a$, $y_{n}\downarrow a$, $u_{n}\uparrow b$ and $v_{n}\downarrow b$. Then $x_{n}+u_{n}\le a_{n}+b_{n}\le y_{n}+v_{n}$, $x_{n}+u_{n}\uparrow a+b$ and  $y_{n}+v_{n}\downarrow a+b$; thus $(a_n+b_n)$ order converges to $a+b$.

To prove (ii) let us first suppose that $\mu_n:=\lambda_{n}-\lambda\ge 0$ for every $n$.
Observing that 
 \[x_{n}-a\leq \mu_n(x_{n}-a)\leq\mu_n(a_{n}-a)\leq\mu_n(y_{n}-a)\leq y_{n}-a\,,\] 
we deduce that $\bigl(\mu_n(a_{n}-a)\bigr)_{n\in\N}$ order converges to $0$.  The  additional assumption that $\bigl(\frac{1}{n}a\bigr)_{n\in \N}$ order converges to $0$ implies that there exist sequences $(s_n)_{n\in\N}$ and $(t_n)_{n\in\N}$ satisfying $s_n\le \frac{1}{n}a\le t_n$ for every $n\in\N$, $s_n\uparrow 0$ and $t_n\downarrow 0$.  Observing that
\[s_n-t_n\le n\mu_n(s_n-t_n)\le\mu_n a\le n\mu_n(t_n-s_n)\le t_n-s_n\,,\]
we deduce that $(\mu_n a)_{n\in\N}$ order converges to $0$.  Thus $\lambda_n a_n=\mu_n(a_n-a)+\mu_n a+\lambda a_n$ order converges to $\lambda a$.

If $\lambda_{n}-\lambda<0$ for every $n$, then the above implies that $(-\lambda_{n}a_n)_{n\in\N}$ order converges to $-\lambda a$ and thus $(\lambda_na_{n})_{n\in\N}$ order converges to $\lambda a$.
\end{proof}

Let us recall that a linear functional $f$ on $X$ is  said to be \emph{positive}  if $x\ge 0$ implies $f(x)\ge 0$.  If $f(x)> 0$  for every nonzero positive element  $x$ of $X$ then $f$ is said to be a \emph{faithful} positive linear functional. A linear functional $f$ is said to be \emph{normal} (or order continuous) if $f(x_\gamma)\longrightarrow f(x)$ whenever $x_\gamma\xrightarrow{o}x$ in $X$.  Clearly, a positive linear functional $f$ on $X$ is normal if and only if $x_\gamma\downarrow 0$ implies $f(x_\gamma)\downarrow 0$.

In the proof of the following proposition we use the fact that an ordered vector space $X$ is monotone order separable if and only if for every net $(x_\gamma)_{\gamma\in\Gamma}$ in $X$ satisfying $x_\gamma\downarrow 0$ there exists an increasing sequence $(\gamma_n)_{n\in\N}$ in $\Gamma$ such that $\bigwedge_{n\in\N} x_{\gamma_n}=0$.

\begin{prop}\label{1.4}
Let $X$ be a conditional monotone $\sigma$-complete  ordered vector space admitting a faithful  normal  positive linear functional $f$. Then $X$ is monotone order separable  and therefore $\tau_{os}(X)=\tau_o(X)$.
\end{prop}
\begin{proof}  Let $(x_\gamma)_{\gamma\in \Gamma}$ be a net in $X$ satisfying $x_\gamma\downarrow 0$. The normality of $f$ implies that $f(x_\gamma)\longrightarrow 0$.  Thus we can select an increasing sequence $(\gamma_n)_{n\in\N}$ in $\Gamma$ such that
$f(x_{\gamma_n})\longrightarrow 0$. Then $s:=\inf_{n\in\N} x_{\gamma_n}\ge 0$ and by the normality of $f$ we deduce that $f(s)=\lim_n f(x_{\gamma_n})=0$. Faithfulness of $f$ implies $s=0$.
\end{proof}

For Riesz spaces the mere  existence of a  faithful positive linear functional (without assuming normality) is sufficient for
the order topology and the sequential order topology to coincide.

\begin{prop}\label{1.5}
Let $X$ be a Riesz space admitting a faithful positive linear functional $f$.
Then $X$ is monotone order separable   and therefore $\tau_{os}(X)=\tau_o(X)$.
\end{prop}
\begin{proof} Let $(x_\gamma)_{\gamma\in \Gamma}$ be a  net in $X$ satisfying $x_\gamma\downarrow 0$. 
  Set $\alpha:=\inf_{\gamma\in \Gamma}f(x_\gamma)$.
  Note that for all $\gamma, \gamma '\in\Gamma$ there is a $\gamma''\in\Gamma$ with $\gamma''\geq\gamma$ and $\gamma''\geq\gamma'$, hence $f(x_\gamma\wedge x_{\gamma'})\geq f(x_{\gamma''})\geq \alpha$. Choose an increasing sequence $(\gamma_n)_{n\in\N}$ in $\Gamma$ such that $f(x_{\gamma_n})\longrightarrow \alpha$.
We  show that $0=\bigwedge_{n\in\N} x_{\gamma_n}$. To this end let $x$ be a lower bound of  $\set{x_{\gamma_n}}{n\in\N}$ and let $\gamma\in \Gamma$. Then
\[0\leq x\vee x_\gamma-x_\gamma\,\le\, x_{\gamma_n}\vee x_\gamma -x_\gamma\,=\,x_{\gamma_n}-x_{\gamma_n}\wedge x_\gamma\,,\]
and therefore
\[0\leq f(x\vee x_\gamma-x_\gamma)\le f(x_{\gamma_n})-f(x_{\gamma_n}\wedge x_\gamma)\leq f(x_{\gamma_n})-\alpha \longrightarrow 0\]
 as $n \to\infty$, i.e. $f(x\vee x_\gamma-x_\gamma)=0$.  
Faithfulness of $f$ implies $x_\gamma=x\vee x_\gamma$ and so $x\le x_\gamma$. We conclude that $x$ is a lower bound for  $\set{x_\gamma}{\gamma\in \Gamma}$.
Consequently $x\leq \inf_{\gamma\in \Gamma}x_\gamma=0$. It follows that $\bigwedge_{n\in\N}x_{\gamma_n}=0$. This proves that  $X$ is monotone order separable and hence by  \propref{1.2}  we deduce that  $\tau_o(X)=\tau_{os}(X)$.
\end{proof}

Our main interest in this paper will be the ordered vector space $X$ given by the self-adjoint part of a von Neumann algebra. In general this is far from being a Riesz space.  However, it is interesting to note that in this case the thesis of Propositions \ref{1.4} and \ref{1.5} holds under the hypothesis that $X$ admits a faithful positive linear functional.   Indeed, if $X$ admits a faithful positive linear functional then any family of pairwise orthogonal projections is necessarily countable, i.e. the corresponding  von Neumann algebra must be $\sigma$-finite. As such it must admit a faithful normal positive linear functional and therefore Proposition~\ref{1.4} applies.

\subsection{Preliminaries on von Neumann algebras}     We first recall a few notions and fix the notation.  We refer to \cite{Blackadar,Sakai,Takesaki} 
for more details.
Let us recall that a \Ca{} $A$ is a complex Banach $\ast$-algebra satisfying $\Vert x^\ast x\Vert=\Vert x\Vert^2$ for every $x\in A$.  We denote by $A_{sa}$ the self-adjoint part of $A$, that is,  $A_{sa}=\set{x\in A}{x=x^\ast}$.
$A_{sa}$ is a real vector space and when endowed with the partial order $\leq$ induced by the cone $A^+:=\set{x^\ast x}{x\in A}$ it gets the structure of an ordered vector space.   In general $A_{sa}$ is far from being a Riesz space.  In \cite{Sherman} it is shown that if $A_{sa}$ is a lattice then $A$ is abelian.  Let $A^1$ denote the closed unit ball of $A$ and let $A_{sa}^1:=A_{sa}\cap A^1$.  An element $p$ of a \Ca{} is called a projection if $p=p^\ast=p^2$.   A \Ca{} may have no non-trivial projections. A linear functional $\varphi$ on $A$ is positive (resp. faithful) if $\varphi|A_{sa}$ is positive (resp. faithful) in the sense of subsection \ref{s2.1}.

A von Neumann algebra $M$ is a \Ca{} that is simultaneously a dual as a Banach space.  In this case $M$ is the dual of a unique Banach space -- called the predual of $M$ and denoted by $M_\ast$.  A linear functional $\varphi$ on $M$ is normal if $\varphi|\saM$ is normal
in the sense described for ordered vector spaces\footnote{Note that this is equivalent to requiring that $\varphi(x_\gamma)\rightarrow \varphi(x)$ for every  net $(x_\gamma)$ in $M_{sa}$ satisfying $x_\gamma\uparrow x$.  This follows because every normal linear functional can be expressed as a linear combination of four normal positive linear functionals.}.  It is known that we can identify the elements of $M_\ast$ with the normal linear functionals in the continuous dual $M^\ast$.  The set of normal positive linear functionals on $M$ is denoted by $M_\ast^+$.  $M$ has always an identity element  $\mathds 1$ and this element is an order-unit for $\saM$.   We recall that $(\saM,\leq)$ is  conditional monotone order complete.  A von Neumann algebra is always rich in projections.  In fact,  a von Neumann algebra is the closure of the span of its projections.     
The set $P(M)$ of all projections in $M$ is a complete orthomodular lattice under the partial order $\leq$ inherited from $\saM$.  
$M$ is called $\sigma$-finite  if every set of nonzero pairwise orthogonal projections in $M$  is at most countable. $M$ is $\sigma$-finite if and only if it admits a faithful normal positive linear functional. 
 For a Hilbert space $H$ we denote by $B(H)$ the \vNa{} of all bounded operators acting on $H$.

For the rest of the paper $M$ is always a von Neumann algebra.  We shall primarily consider the order topology (and the corresponding sequential variant) on the following three posets: $\saM$\,, $\saM^1$ and $P(M)$.  We  shall study the properties of the order topology of these posets, compare the order topology to other standard locally convex topologies on $M$  and relate the properties of the order topology to the underlying algebraic structure of $M$.

We recall that the weak$^\ast$-topology $\sigma(M,M_\ast)$ on $M$
is the coarsest locally convex topology compatible with the duality $\langle M_\ast,M\rangle$.  The finest locally convex topology on $M$ compatible with this duality is the Mackey topology
 $\tau(M,M_\ast)$.
Lying between these topologies we have the $\sigma$-strong topology $s(M,M_\ast)$ determined by the family of seminorms $\set{\ro_\psi}{\psi\in M_{\ast}^+}$ where $\ro_{\psi}(x)=\sqrt{\psi(x^\ast x)}$ and the $\sigma$-strong$^\ast$ topology $s^\ast(M,M_\ast)$
determined by the family of seminorms $\set{\eta_\psi}{\psi\in M_{\ast}^+}$ where $\eta_\psi(x)=\sqrt{\psi(x^\ast x)+\psi(x x^\ast)}$.  
$M$ can be faithfully represented on a Hilbert space $H$, i.e. $M$ can be identified with a subalgebra of  $B(H)$ closed w.r.t. the weak operator topology and therefore one can endow $M$ with the 
 strong operator topology $\tau_s$ and the 
 weak operator topology $\tau_w$.   These are  the topologies of pointwise convergence w.r.t. the norm topology or the weak topology on $H$, respectively. Note however that $\tau_s$ and $\tau_w$ in general depend on the particular representation.
It is well known  that
\begin{equation}\label{e0}
\begin{split}
\sigma(M,M_\ast)\subseteq s(M,M_\ast)\subseteq s^\ast(M,M_\ast) \subseteq \tau(M,M_\ast)\,,\\
\tau_w \subseteq \tau_s\,,\qquad \tau_w\subseteq \sigma(M,M_\ast)\,,\qquad\tau_s\subseteq  s(M,M_\ast)\,.
\end{split}
\end{equation} 

By the uniform boundedness principle it follows that if $A$ is a set of bounded linear operators on a Hilbert space that is bounded w.r.t. the weak operator topology then $A$ uniformly bounded.  Hence, in view of \eqref{e0} a subset $K$ of $M$ that is bounded w.r.t. any of the above locally convex topologies is uniformly bounded.  Furthermore, we recall that if $x,y\in\saM$  then: {\rm(i)}~$-y\le x\le y$ implies $\|x\|\le \|y\|$; and {\rm(ii)}~$-\|x\|\mathds 1\le x \le\|x\|\mathds 1$; i.e. if $K\subseteq \saM$ then $K$ is bounded (w.r.t. any of the above locally convex topologies) if and only if it is order bounded.

On  bounded parts of $M$ the $\sigma$-strong topology coincides with the strong operator topology and the weak$^\ast$ topology coincides with the weak operator topology.  Moreover, a deep classical result by C.~Akemann \cite{Akemann} says that 
\begin{equation}\label{Ake}
s^\ast(M,M_\ast)|K=\tau(M,M_\ast)|K
\end{equation}
for every bounded subset $K$ of $M$.
 Since $s(M,M_\ast)$ and $s^\ast(M,M_\ast)$ coincide on $\saM$, it follows that 
\begin{equation}\label{e1}
\tau_s|K=s(M,M_\ast)|K=\tau(M,M_\ast)|K.
\end{equation}
for every bounded subset $K$ of $\saM$.

Let $\tau_u$ denote the uniform topology (i.e. $\|\cdot\|$-topology) on $M$.  We show that $\tau_u|\saM$ is finer than the sequential order topology (and hence than the order topology) of $\saM$.  Suppose that $(x_n)_{n\in\N}$ is a sequence in $\saM$ such that $\Vert x_n\Vert\to 0$.  If we set $\lambda_n:=\sup_{k\ge n}\Vert x_k\Vert$ then
\[-\lambda_n\mathds 1\le-\Vert x_n\Vert\mathds 1\le x_n\le \Vert x_n\Vert\mathds 1\le \lambda_n\mathds 1\]
and $\lambda_n\mathds 1\downarrow 0$ in $\saM$, i.e. $x_n\xrightarrow{o} 0$ in $\saM$.

We shall now compare the order topology $\tau_o(\saM)$ with the $\sigma$-strong topology $s(M,M_\ast)$.  If $(x_\gamma)_{\gamma\in\Gamma}$ is a net in $\saM^+$ and $x_\gamma\xrightarrow{o} 0$ in $\saM$ then $\psi(x_\gamma)\to 0$ for every $\psi\in M_\ast^+$.  The Cauchy-Schwartz inequality yields $\psi(x_\gamma^2)\le\sqrt{\psi(x_\gamma)\psi(x_\gamma^3)}$
 and therefore one obtains $x_\gamma\st{M}0$ by observing that the net $(x_\gamma)_{\gamma\in\Gamma}$ is eventually bounded. 
 Now suppose that  $y_\gamma\xrightarrow{o} y$ 
in $\saM$.  Let $(a_\gamma)_{\gamma\in\Gamma}$ and $(b_\gamma)_{\gamma\in\Gamma}$ be nets in $\saM$ such that $a_\gamma\le y_\gamma\le b_\gamma$, $a_\gamma\uparrow y$ and $b_\gamma\downarrow y$.  Then $y_\gamma-a_\gamma\ge 0$ for every $\gamma\in\Gamma$ and $y_\gamma-a_\gamma\xrightarrow{o} 0$.  The above observation implies that  $y_\gamma-a_\gamma\st{M} 0$ and $y-a_\gamma\st{M} 0$.  The linearity of $s(M,M_\ast)$ implies that $y_\gamma\st{M} y$.  Thus, we conclude that
 \begin{equation}\label{e2}
 s(M,M_\ast)|\saM\subseteq \tau_o(\saM).
 \end{equation}
 In particular $\tau_o(\saM)$ is Hausdorff and $\saM^1$ is $\tau_o(\saM)$-closed.  The inclusion in \eref{e2} together with the equality of \eref{e1} imply that $\tau(M,M_\ast)|K\subseteq\tau_o(\saM)|K$ for every bounded subset $K$ of $\saM$.  Using the fact that an order convergent net of $\saM$ is eventually bounded, it is easy to see that a subset $X$ of $\saM$ is closed w.r.t. $\tau_o(\saM)$ if and only if $X\cap r\saM^1$ is closed w.r.t. $\tau_o(\saM)$ for every $r>0$.  Hence, if $X\subseteq \saM$ is $\tau(M,M_\ast)$-closed then $X\cap r\saM^1$ is $\tau(M,M_\ast)$-closed and therefore one obtains that $X\cap r\saM^1$ is $s(M,M_\ast)$-closed applying (\ref{e1}) to the $s(M,M_\ast)$-closed set $K:=r\saM^1$.
Then \eref{e2} implies that $X\cap r\saM^1$ is $\tau_o(\saM)$-closed.  This holds for every $r>0$ and therefore the following inclusion  follows
 \begin{equation}\label{e3}
 \tau(M,M_\ast)|\saM\subseteq \tau_o(\saM).
 \end{equation}

We summarize the above observations in {\rm(2.\ref{eq0})} of the following proposition.  Since $\saM^1$ and $P(M)$ are $s(M,M_\ast)$-closed,  {\rm (2.\ref{eq1})--(2.\ref{eq3})}follow by \eref{e2} and \propref{1.6}{\rm(i)}.

 \begin{prop}\label{1.7} The following inclusions hold.
  \begin{enumerate}[{\rm({2.}1)}]
 \setcounter{enumi}{5}
\item\label{eq0} $s(M,M_\ast)|\saM\subseteq\tau(M,M_\ast)|\saM\subseteq \tau_o(\saM)\subseteq\tau_{os}(\saM)\subseteq\tau_u|\saM$
\item\label{eq1} $\tau_o(\saM)|\saM^1\subseteq\tau_o(\saM^1)$
\item\label{eq2}  $\tau_o(\saM^1)|P(M)\subseteq \tau_o(P(M))$
\item\label{eq3} $\tau_o(\saM)|P(M)\subseteq\tau_o(P(M))$
\end{enumerate}
\end{prop}

\begin{lem}\label{4.2}
 Let $0\le x\le \mathds 1$ in $M$.
\begin{enumerate}[{\rm(i)}]
\item  If $p$ is a projection, then  $x\ge p$ if and only if   $px=xp=p$.
 \item If  $\set{p_\lambda}{\lambda\in\Lambda}$ is a set in $P(M)$ and $x\in\saM^1$ satisfies $x\ge p_\lambda$ for every $\lambda\in\Lambda$ then $x\ge p$ where $p=\bigvee_{\lambda\in\Lambda} p_\lambda$ in $P(M)$.
\end{enumerate}
\end{lem}
\begin{proof} Let us suppose that $M$ acts on a Hilbert space $H$.\\ 
{\rm(i)}~If $x\ge p$ then for any unit vector $\xi$ in  $H$  that lies in the range of projection $p$ we have $1\ge (x\xi, \xi)\ge (\xi, \xi)=1$.
So by the Cauchy-Schwartz inequality we deduce that $x\xi=\xi$.
Consequently, $xp=p.$
Conversely,  if $xp=p$ then $p$ and $x$ commute and therefore $(\mathds 1-p)x=(\mathds 1-p) x (\mathds 1-p)\ge 0$.  Hence $x=p+(\mathds 1-p)x \ge p.$\\
{\rm(ii)}~If $x\ge p_\lambda$ for every $\lambda\in\Lambda$ then $x\xi=\xi$ for every vector $\xi$ in the range of $p$.  Hence $xp=p=px$ and thus $x\ge p$ by {\rm(i)}.
\end{proof}

\begin{prop}\label{p1} The following  statements are equivalent:
\begin{enumerate}[{\rm(i)}]
\item $M$ is abelian,
\item $\tau_o(\saM)|P(M)=\tau_o(\saM^1)|P(M)$,
\item $\tau_o(\saM)|P(M)=\tau_o(P(M))$,
\item $\tau_o(\saM)|\saM^1=\tau_o(\saM^1)$.
\end{enumerate}
\end{prop}
\begin{proof}
When $M$ is abelian $(\saM,\le)$ is a  Dedekind complete lattice and since $\saM^1$ and $P(M)$ are $s(M,M_\ast)$-closed sublattices of $\saM$, it follows by (\ref{e2}) and \propref{1.6} (ii) that $\tau_o(\saM)|\saM^1=\tau_o(\saM^1)$ and $\tau_o(\saM)|P(M)=\tau_o(\saM^1)|P(M)=\tau_o(P(M))$.

When $M$ is not abelian it contains a von Neumann subalgebra $N$ (not necessarily unital) that is $\ast$-isomorphic to $B(H_2)$ where $H_2$ is a two-dimensional Hilbert space.  We will identify $N$ with $B(H_2)$.  We show that $\tau_o(M_{sa}^1)|P(N)$ is discrete.  To this end we suppose that $(p_\gamma)_{\gamma\in\Gamma}$ is a net of projections in $N$ that order converges in $(\saM^1,\le)$, say to $p$.  (Note that $p$ is a also a projection in $N$ because $P(N)$ is $s(M,M_\ast)$-closed and order convergence in $\saM^1$ implies convergence w.r.t. $s(M,M_\ast)$.) Suppose, for contradiction, that $(p_\gamma)_{\gamma\in\Gamma}$ is not eventually constant. The inclusions
\[\tau_o(\saM)|N_{sa}\supseteq s(M,M_\ast)|N_{sa}=\tau_u|N_{sa}
\]
and {\rm ({2.}\ref{eq1})} imply that $p_\gamma\xrightarrow{\tau_u}p$ and therefore $p\notin\{0,\mathds 1_N\}$.  We can thus assume that the range of $p_\gamma$ is one-dimensional for every $\gamma\in\Gamma$.  Lemma \ref{4.2} implies that if $x\in M_{sa}^1$ satisfies $x\ge p_\gamma$ for every $\gamma\ge \gamma'$ then $x\ge \bigvee_{\gamma\ge \gamma'}p_\gamma=\mathds 1_N$.  This implies that  $p=\mathds 1_N$, a contradiction.  Thus, every subset of $P(N)$ is $\tau_o(\saM^1)$-closed, i.e. $\tau_o (\saM^1)|P(N)$ is discrete.  On the other-hand, observe that $\tau_o(\saM)|P(N)\subseteq \tau_u|P(N)$, i.e. $\tau_o(\saM)|P(N)$ is not discrete.  Thus, we have proved that if {\rm(ii)} is true then $M$ is abelian.  

If {\rm(iii)} is true, then we combine with ({\rm 2.\ref{eq1}}) and ({\rm 2.\ref{eq2}}) to obtain
\[\tau_o(P(M))=\tau_o(\saM)|P(M)\subseteq\tau_o(\saM^1)|P(M)\subseteq\tau_o(P(M)),\]
i.e. {\rm(iii)} implies {\rm(ii)}.  The implication {\rm(iv)} $\Rightarrow$ {\rm(ii)} is trivial.

\end{proof}

Proposition \ref{p1} implies that the inclusion in {\rm (2.\ref{eq1})} and {\rm (2.\ref{eq3})} are proper for nonabelian von Neumann algebras .  In contrast,  in the proof of Proposition \ref{p1} it is shown that when $M=B(H_2)$ then the inclusion in {\rm(2.\ref{eq2})} is an equality.  The question on when we get an equality in {\rm(2.\ref{eq2})} will be dealt with in Section \ref{section4};  in fact, we shall prove that for  $\sigma$-finite von Neumann algebras this characterizes finiteness.

\begin{rem}\label{r1}
Let $M$ have an infinite linear dimension; then it contains a sequence of pairwise orthogonal projections $(p_n)_{n\in\N}$:

\noindent {\rm(i)}~Using the fact that every order convergent net is eventually bounded it is easy to see that the set $\set{\sqrt{n}p_n}{n\in\N}$ is closed w.r.t. $\tau_o(\saM)$.  On the other-hand $0$ lies in the $s(M,M_\ast)$-closure of $\set{\sqrt{n}p_n}{n\in\N}$.  So $s(M,M_\ast)|\saM\subsetneq \tau_o(\saM)$.

\noindent {\rm(ii)}~The sequence $(p_n)_{n\in\N}$ satisfies $\limsup_n p_n=\liminf_n p_n=0$, i.e. it order converges to $0$ in $(P(M),\le)$.  Thus {\rm {2.}(\ref{eq3})} implies that $(kp_n)_{n\in\N}$ converges to $0$ w.r.t. $\tau_o(\saM)$ for every $k\in\N$.  For every $\tau_o(\saM)$-neighbourhood  $U$ of $0$ there exists $n(k,U)\in \N$ such that $kp_n\in U$ for every $n\ge n(k,U)$.  Define
    \[\mathcal N:=\set{(k,U)}{k\in\N,\,U\text{ is a }\tau_o(\saM)\text{-neighbourhood of }0}\,\]
    and equip it with the partial order defined by $(k_1,U_1)\le (k_2,U_2)$ if and only if $k_1\le k_2$ and $U_2\subseteq U_1$.
    Then $\mathcal N$ is an upward directed set.  We can define a net $\bigl(x_{(k,U)}\bigr)_{(k,U)\in\mathcal N}$ by setting
$x_{(k,U)}:=kp_{n(k,U)}$.  It is clear that this net is not eventually bounded despite being convergent to $0$ w.r.t. $\tau_o(\saM)$. Observe further that no subnet of this net is eventually bounded and therefore no subnet is order convergent in $\saM$.
\end{rem}

 In contrast to the  example exhibited in {\rm(ii)} of the previous remark let us observe that any sequence converging in the order topology is bounded.  
 Item {\rm(ii)} of the previous remark suggests (particularly in view of Proposition \ref{1.1}) that a favoured case occurs when the sequential order topology coincides with the order topology because in this case -- at least for sequences -- convergence w.r.t. the order topology can be described by order convergent subsequences.  The following proposition says that this occurs precisely when $M$ is $\sigma$-finite.

\begin{prop}\label{2.1}
The following three statements are equivalent:
 \begin{enumerate}[{\rm(i)}]
 \item   $M$ is $\sigma$-finite,
 \item $\tau_{os}(M_{sa})=\tau_{o}(M_{sa})$,
 \item $\tau_{os}(M_{sa}^1)=\tau_{o}(M_{sa}^1)$,
 \item $\tau_{os}(P(M))=\tau_{o}(P(M))$. 
\end{enumerate} 
\end{prop}
\begin{proof} We recall that bounded monotonic nets in $\saM$ converge  w.r.t. $s(M,M_\ast)$ to their supremum/infimum.  Using that $\saM^1$ and $P(M)$ are $s(M,M_\ast)$-closed, it follows that if $\saM$ is monotone order separable then $\saM^1$ is monotone order separable; and if $\saM^1$ is monotone order separable then $P(M)$ is monotone order separable.
If $M$ is $\sigma$-finite then it admits a faithful normal positive linear functional and
so, by Proposition~\ref{1.4},  {\rm(i)}$\Rightarrow${\rm(ii)}$\Rightarrow${\rm(iii)}$\Rightarrow${\rm(iv)}.
When $M$ is not $\sigma$-finite, $P(M)$ contains an uncountable family of  nonzero orthogonal projections and thus it  is not monotone order separable, i.e. {\rm(iv)}$\Rightarrow${\rm(i)}.
\end{proof}

\section{Vector spaces with mixed topology}

 Now we consider the mixed topology on a vector space introduced and studied in detail in \cite{Wiweger}. We first list some of its basic known properties and then we add some new facts needed in the sequel.

 In this section let $X$ be a real vector space endowed with two linear Hausdorff topologies  $\tau$ and $\tau'$.
For each sequence $(U'_n)_{n\in\N}$ of $0$-neighbourhoods in $(X,\tau')$ and for each $0$-neighbourhood  $U$ in $(X,\tau)$   define
\[\gamma\bigl((U'_n)_{n\in\N},U\bigr):=\bigcup_{n\in\N}\sum_{i=1}^n (U'_i\,\cap\,iU).\]
Then the family of these sets 
is a basis  of $0$-neighbourhoods for some linear Hausdorff topology $\gamma[\tau,\tau']$ called  the \emph{mixed topology} determined by $\tau$ and $\tau'$.  It is clear that if $X$ is a complex vector space and the Hausdorff topologies $\tau$ and $\tau'$ are linear over $\mathds C$ then $\gamma[\tau,\tau']$ is also linear over $\mathds C$.

\begin{prop}\cite[2.1.1]{Wiweger}\label{W2.1}\begin{enumerate}[{\rm(i)}]
\item $\tau'\subseteq \gamma[\tau,\tau']$
\item If $\tau'\subseteq\tau$ then $\gamma[\tau,\tau']\subseteq \tau$.
\item If $\tau$ and $\tau'$ are locally convex, then $\gamma[\tau,\tau']$ is locally convex.
\end{enumerate}
\end{prop}

\begin{prop}\label{W2.2} \cite[2.2.1,2.2.2]{Wiweger}
\begin{enumerate}[{\rm(i)}]
\item $\gamma[\tau,\tau']|Z=\tau'|Z$ for every $\tau$-bounded subset $Z$ of $X$.
\item If $(X,\tau)$ is locally bounded, then $\gamma[\tau,\tau']$ is the finest of all linear topologies agreeing with $\tau'$ on every $\tau$-bounded subset of $X$.
\end{enumerate}
\end{prop}

\propref{W2.2} {\rm(ii)} implies that $\gamma[\tau,\tau_1]=\gamma[\tau,\tau_2]$ when $(X,\tau)$ is locally bounded and $\tau_1$ and $\tau_2$ are Hausdorff linear topologies on $X$ such that $\tau_1|Z=\tau_2|Z$ for every $\tau$-bounded subset $Z$; in particular \[\gamma[\tau,\tau']=\gamma\bigl[\tau,\gamma[\tau,\tau']\bigr].\]

\begin{prop}\label{W2.4}\cite[2.4.1]{Wiweger}   If $\Vert\cdot\Vert$ is a norm on $X$ inducing $\tau$ and the unit ball of $(X,\Vert\cdot\Vert)$ is $\tau'$-closed then a set $A\subseteq X$ is $\gamma[\tau,\tau']$-bounded if and only if it is simultaneously $\Vert\cdot\Vert$-bounded and $\tau'$-bounded.
\end{prop}

The following two theorems will be of great use in Section \ref{section3} 
 
\begin{thm}\label{2.5}
Assume that  
\begin{enumerate}[{\rm(i)}]
\item $\tau$ is induced by a norm $\Vert\cdot\Vert$ on $X$,
\item the unit ball $X^1$ of $(X,\Vert\cdot\Vert)$ is $\tau'$-closed, but not $\tau'$-compact  and
\item $\tau'|X^1$ is metrizable and strictly coarser than $\tau|X^1$.
\end{enumerate} 
Then $\bigl(X,\gamma [\tau,\tau']\bigr)$ is not a sequential space.
\end{thm}

\begin{proof}
$X^1$ contains by (ii) a sequence $(a_n)_{n\in\N}$ without $\tau'$-cluster point. By (iii) there is an integer  $m_0>1$ and a sequence $(b_n)_{n\in\N}$ in $X^1$ converging to $0$ w.r.t. $\tau'$ such that $\|b_n\| > 1/m_0$ for $n\in\N$.

We will show that
\[ F:=\left\{\frac{1}{m}a_n+m b_n:n,m\in\N,m\ge m_0\right\}.\]
is sequentially closed, but not closed in $\bigl(X,\gamma [\tau,\tau']\bigr)$.

$0\notin F$ since $\|\frac{1}{m}a_n\|<1 < \|m b_n\|$ for $n,m\in\N$ with $m\geq m_0$.
We show that on the other hand $0$ is a  $\gamma [\tau,\tau']$-limit point of $F$.  Let  $W:=\gamma\bigl((U_k')_{k\in\N},U\bigr)$ be a $0$-neighbourhood in  $\gamma[\tau,\tau']$ where $U_k'$ and $U$ are $0$-neighbourhoods in $\tau'$ and $\tau$, respectively.  Since $B:=\bigcup_{n\in\N}\{a_n,b_n\}\subseteq X^1$ and $\tau'\subseteq\tau$, we have $\frac{1}{m}B\subseteq U_1'\cap U$ for some $m\ge m_0$.  Then $\frac{1}{m}a_n\in U_1'\cap U$ for every $n\in\N$.  Now, let $l\in\N$ such that $\frac{m}{l}B\subseteq U$.  Then $mb_n\in lU$ for every $n\in\N$.  Since $(mb_n)_{n\in\N}$ converges to $0$ w.r.t. $\tau'$, there exists $n_0\in\N$ such that $mb_{n_0}\in U'_l$.  Then $\frac{1}{m}a_{n_0}+mb_{n_0}\in U_1'\cap U\,+\,U_l'\cap lU\subseteq W$

We now show that $F$ is sequentially closed  w.r.t. $\gamma[\tau,\tau']$. 
 Let $(g_j)_{j\in\N}$ be a sequence in $F$ converging to $g$ w.r.t. $\gamma[\tau,\tau']$.  We can write $g_j=\frac{1}{m_j}a_{n_j}+m_jb_{n_j}$ where $m_j,n_j\in\N$ and $m_j\ge m_0$.  The set $\{g_j:j\in\N\}$  is $\gamma[\tau,\tau']$-bounded and therefore $\tau$-bounded in virtue of Proposition \ref{W2.4}.  Hence $\{m_j b_{n_j}:j\in\N\}$ is $\tau$-bounded. But since $\|b_n\|\geq 1/m_0$ for all ${n\in\N}$ this can only happen if $\{m_j:j\in\N\}$ is finite.  Thus, passing to a subsequence, we may assume that $m_j$ is constant $(=m)$, i.e. $g_{n_j}=\frac{1}{m}a_{n_j}+m b_{n_j}$.  Suppose that $\{n_j:j\in\N\}$ is not finite.  Passing to a subsequence we may assume that $n_j$  are strictly increasing.  Since $b_{n_j}\xrightarrow{\tau'} 0$ and $\tau'\subseteq \gamma[\tau,\tau']$ we deduce that
\[a_{n_j}=mg_{n_j}-m^2 b_{n_j}\xrightarrow{\tau'}mg\quad\text{as }j\to\infty,\]
in contradiction to the fact that $(a_n)_{n\in\N}$ has no $\tau'$-cluster point.  Therefore $\{n_j:j\in\N\}$ is finite.  But this implies that $g$ belongs to $F$.
\end{proof}

\begin{thm}\label{2.6} Let $\tau'$ be induced by a pointwise bounded family $\{\rho_\lambda:\lambda\in\Lambda\}$ of seminorms on $X$  and $\tau$ be the topology induced by the norm
\[\Vert x\Vert:=\sup_{\lambda\in\Lambda}\rho_\lambda(x).\]
Assume further that the unit ball $X^1$ of $(X,\Vert\cdot\Vert)$ is $\tau'$-compact.

Then a subset $C$ of $X$ is $\gamma[\tau,\tau']$-closed if and only if $C\cap\,rX^1$ is $\gamma[\tau,\tau']$-closed for every $r>0$.
\end{thm}

\begin{proof}
Let $C$ be $\gamma[\tau,\tau']$-closed and $r>0$. Since $rX^1$ is $\tau'$-compact, it is $\tau'$-closed, hence $\gamma[\tau,\tau']$-closed since $\tau'\subseteq\gamma[\tau,\tau']$. Therefore  $C\cap\,rX^1$ is $\gamma[\tau,\tau']$-closed.

The proof of the reverse implication is based on the following two lemmas. We use therein the notation 
$$B(f):=\{x\in X: \rho_\lambda(x)\leq f(\lambda) \text{ for all } \lambda\in\Lambda\} \text{  if  } f:\Lambda\rightarrow (0,+\infty].$$
Moreover, since the seminorms $\rho_\lambda$ are not assumed to be different, we may assume that $\Lambda$ is infinite. 

\begin{lem}\label{l1}
Let  $f:\Lambda\rightarrow (0,+\infty)$ and $\sup_{\lambda\in\Lambda}f(\lambda)\leq s<+\infty$.  Assume that  $A\subseteq X$ such that $A\cap B(f)=\emptyset$ and $A\cap sX^1$ is $\tau'$-closed.

Then there exists a finite subset $F$ of $\Lambda$  such that $A\cap B(g)=\emptyset$ where $g(\lambda)=f(\lambda)$ for $\lambda\in F$ and
$g(\lambda)=s$ for $\lambda\in \Lambda\setminus F$.
\end{lem}
\begin{proof}
Otherwise for any finite subset $F$ of $\Lambda$  there exists  $x_F\in A$ with $\rho_\lambda(x_F)\leq f(\lambda)$ for $\lambda\in F$ and $\rho_\lambda(x_F)\leq s$ for all $\lambda\in\Lambda$. Since $A\cap sX^1$ is $\tau'$-compact, $(x_F)_{F\subseteq\Lambda,|F|<\infty}$ has a subnet $\tau'$-converging to an element $x\in A\cap sX^1$. Moreover $\rho_\lambda (x_F)\leq f(\lambda)$ eventually. Therefore $\rho_\lambda (x)\leq f(\lambda)$ for all $\lambda\in\Lambda$. It follows $x\in A\cap B(f)$, a contradiction.
\end{proof}

\begin{lem}\label{l2}
Let $A\subseteq X\setminus\{0\}$ be such that  $A\cap rX^1$ is $\tau'$-closed for every $r\in\R$, $r>0$. Then there is a sequence $(\lambda_n)_{n\in\N}$ in $\Lambda$ and a real sequence $(a_n)_{n\in\N}$ with $0<a_n\uparrow +\infty$
such that $A\cap B(g)=\emptyset$ where $g(\lambda_n)=a_n$ for $n\in\N$ and $g(\lambda)=+\infty$ otherwise.
\end{lem}
\begin{proof}
By assumption $A\cap X^1$ is $\tau'$-closed, therefore $\tau$-closed since $\tau'\subseteq\tau$.
Hence there is an $\varepsilon>0$ such that $A\cap \varepsilon X^1=\emptyset$.

Let $F_0:=\emptyset$.  We define inductively a strictly increasing sequence $(F_n)_{n\in\N}$ of finite subsets of $\Lambda$ such that $A\cap B(g_n)=\emptyset$ where $g_n$ is defined by
\begin{align*}
g_n(\lambda)=\begin{cases}
i\varepsilon\qquad &\text{if }\lambda \in F_{i}\setminus F_{i-1} \text{ and } 1\le i\le n,\\
(n+1)\varepsilon\qquad&\text{if }\lambda\in\Lambda\setminus F_n.
\end{cases}
\end{align*}

By the choice of $\varepsilon$, $A\cap B(g_n)=\emptyset$  is satisfied for $n=0$ defining $g_0(\lambda)=\varepsilon$ for all $\lambda\in\Lambda$.

For the inductive step $[n-1\rightarrow n]$ we apply \lemref{l1} with $f:=g_{n-1}$ and $s:=(n+1)\varepsilon$. Choose $F$ according to  \lemref{l1} and let $F_n$ be a finite subset of $\Lambda$ with $F\cup F_{n-1}\subseteq F_n$ and $F_{n-1}\neq F_n$. If we set $g_n(\lambda)=g_{n-1}(\lambda)$ for $\lambda\in F_n$ and $g_n(\lambda)=(n+1)\varepsilon$ for $\lambda\in \Lambda\setminus F_n$ then $A\cap B(g_n)=\emptyset$.

Let $g:=\sup_{n\in\N} g_n$.  Then $g(\lambda)=n\varepsilon$ whenever  there exists $n\in\N$ such that $\lambda\in F_n\setminus F_{n-1}$.  Otherwise $g(\lambda)=+\infty$.  
Since the sequence $(g_n)$ is increasing, $B(g)=\bigcup_{n\in\N} B(g_n)$ and therefore $A\cap B(g)=0$.

To complete the proof  let $k_n:=|F_n|$, choose a sequence $(\lambda_n)_{n\in\N}$ in $\Lambda$ with $F_n=\{\lambda_i: i\leq k_n\}$ and set $a_i:=g(\lambda_i)$.
\end{proof}

To conclude the proof of Theorem \ref{2.6} first we recall that   the sets $\{x\in X: \rho_{\lambda_n}(x)\leq a_n \text{ for all }n\in\mathbb{N}\}$ where $\lambda_n\in\Lambda$ and  $0<a_n\uparrow +\infty$ form a $0$-neighbourhood base of $\bigl(X,\gamma[\tau,\tau']\bigr)$ (see \cite[Theorem 3.1.1]{Wiweger}).  Suppose that $C\subseteq X$ and $C\cap\,rX^1$ is $\gamma[\tau,\tau']$-closed for every $r>0$.  Since  $\gamma[\tau,\tau']$ and $\tau'$ induce on $rX^1$ the same topology (see \propref{W2.2}{\rm(i)}) and since $rX^1$ is $\tau'$-closed it follows that $C\cap rX^1$ is also $\tau'$-closed.   Let $x\notin C$ and $A:=C-x$. Then $0\notin A$. It follows from \lemref{l2}  that $0$ does not belong to the  $\gamma[\tau,\tau']$-closure $\overline{A}$ of $A$, i.e. $x\notin \overline{C}$.
\end{proof}

\begin{cor}\label{2.7}
Under the assumptions of \thmref{2.6} we have: If $\tau''$ is a (not necessarily linear) topology on $X$  such that $\tau''|rX^1=\tau'|rX^1$ for every $r>0$ then $\tau''\subseteq \gamma[\tau,\tau']$.
\end{cor}
\begin{proof}
Let $C$ be a $\tau''$-closed subset of $X$. In view of \thmref{2.6} it is enough to show that $C\cap rX^1$ is $\gamma[\tau,\tau']$-closed for every $r>0$. Let $(x_\gamma)_{\gamma\in\Gamma}$ be a net in $C\cap rX^1$ converging to $x$ w.r.t. $\gamma[\tau,\tau']$. We show that $x\in C\cap rX^1$. 
First observe that  $(x_\gamma)_{\gamma\in\Gamma}$ converges to $x$ also w.r.t. $\tau'$ since $\tau'\subseteq \gamma[\tau,\tau']$. 
 Moreover, since $rX^1$ is $\tau'$-closed we get $x\in rX^1$.
 Thus, by assumption, $(x_\gamma)_{\gamma\in\Gamma}$ converges to $x$ w.r.t. $\tau''$. Therefore $x\in C$ since $C$ is $\tau''$-closed, and finally $x\in C\cap rX^1$.
\end{proof}

\begin{cor}\label{2.8}
With the assumptions of \thmref{2.6} we have: If $\Lambda$ is countable then $\bigl(X,\gamma[\tau,\tau']\bigr)$ is a sequential space.
\end{cor}

\begin{proof}
Let $C$ be a sequential closed subset of $\bigl(X,\gamma[\tau,\tau']\bigr)$ and let $r>0$. Then since  $rX^1$ is $\tau'$-closed (and therefore $\gamma[\tau,\tau']$-closed) $C\cap rX^1$ is sequential closed w.r.t. $\gamma[\tau,\tau']$. But on the $\tau'$-closed subset $rX^1$ the topology $\tau'$ agrees with $\gamma[\tau,\tau']$.  Therefore $C\cap rX^1$ is sequential closed w.r.t. $\tau'$.  By our assumption on $\Lambda$ we get that $C\cap rX^1$ is closed w.r.t. $\tau'$ and therefore $C\cap rX^1$ is $\gamma[\tau,\tau']$-closed. Hence $C$ is $\gamma[\tau,\tau']$-closed by \thmref{2.6}.
\end{proof}

We now give a first application of  Theorem \ref{2.5} and Corollary \ref{2.8}. Let $(X,\Sigma,\mu)$ be a $\sigma$-finite measure space, $\tau_\infty$ the topology of the Banach space $(L^\infty,\Vert\,\Vert_\infty)$ and $\tau_\mu$
the topology of convergence in measure (on sets of finite measure), i.e. the Hausdorff linear topology induced by the family of $F$-seminorms  $\rho_F:L^\infty\ni [f]\mapsto \int_X |f|\wedge\chi_F\,d\mu $ ($F\in\Sigma$ with $\mu(F)<\infty$ ). (See \cite[Proposition 245A, p. 172]{Fremlin}.) If $\mu$ is not purely atomic, then Theorem \ref{2.5} implies that $(L^\infty,\gamma[\tau_\infty,\tau_\mu])$ is not a sequential space. If $\mu$ is purely atomic, then $L^\infty$ can be identified with the sequence space $\ell^\infty$ and $\tau_\mu$ with the topology of pointwise convergence on $\ell^\infty$, which is generated by the seminorms $p_n:\ell^\infty\ni (x_i)\mapsto x_n$ ($n\in\mathbb{N}$). It follows therefore from Corollary \ref{2.8}  that $(L^\infty,\gamma[\tau_\infty,\tau_\mu])$ is a sequential space. Combining these two results with Novak's result \cite[Theorem 5]{Novak}, saying that the mixed topology $\gamma[\tau_\infty,\tau_\mu]$ coincides with the Makey topology $\tau(L^\infty,L^1)$ on $L^\infty$ induced by the dual pairing $(L^\infty,L^1)$, we obtain:

\begin{thm}\label{abelian}
 Let $(X,\Sigma,\mu)$ be a $\sigma$-finite measure space. Then the space $(L^\infty,\tau(L^\infty,L^1))$ is  sequential  if and only if $\mu$ is purely atomic. 
\end{thm}

This theorem will be generalized in Theorem \ref{3.10} for von Neumann algebras.

\section{The order topology and the sequential order topology on $\saM$}\label{section3}
  
On $M$ we consider the mixed topology $\gamma\bigl[\tau_u,s^\ast(M,M_\ast)\bigr]$ determined by $\tau_u$ and $\sstar$.

\begin{thm}\label{mixed}
The Mackey topology $\tau(M,M_\ast)$ coincides with the mixed topology $\gamma\bigl[\tau_u,s^\ast(M,M_\ast)\bigr]$.
\end{thm}
\begin{proof} Proposition \ref{W2.2}{\rm(ii)} with $\tau:=\tau_u$ and $\tau':=s^\ast(M,M_\ast)$ and Akemann's Theorem (\ref{Ake}) already imply $\tau(M,M_\ast)\subseteq\gamma\bigl[\tau_u,s^\ast(M,M_\ast)\bigr]$.    For the converse observe that if $\varphi$ is a linear functional on $M$ continuous w.r.t.  $\gamma\bigl[\tau_u,s^\ast(M,M_\ast)\bigr]$ then  \cite[Corollary 1.8.10, p. 21]{Sakai} implies that $\varphi$ is $\sigma(M,M_\ast)$-continuous.  Hence the inclusion $\gamma\bigl[\tau_u,s^\ast(M,M_\ast)\bigr]\subseteq \tau(M,M_\ast)$ follows because $\tau(M,M_\ast)$ is the finest locally convex topology on $M$ compatible with the duality $(M,M_\ast)$.
\end{proof}

\begin{rem}Theorem \ref{mixed} is a  generalization of \cite[Theorem 5]{Novak}.   Let $(X,\Sigma,\mu)$ be a localisable measure space.  Then $L^\infty$ is an abelian von Neumann algebra (see \cite[Theorem 243G, p.154]{Fremlin}). 
It is easy to verify that on bounded parts of $L^\infty$ the topology $\tau_\mu$ of convergence in measure
agrees with $s^\ast(L^\infty,L^1)$ $(=s(L^\infty,L^1))$ and therefore $\tau(L^\infty,L^1)=\gamma[\tau_u,\tau_\mu]$ by Theorem \ref{mixed}.  
In \cite[Theorem 5]{Novak} it is shown that when $(X,\Sigma,\mu)$ is $\sigma$-finite then $\tau(L^\infty,L^1)=\gamma[\tau_u,\tau_\mu]$.
\end{rem}

\begin{thm}\label{3.4} Let $(x_n)_{n\in\N}$ be a sequence in $\saM$ and let $x\in\saM$.
\begin{enumerate}[{\rm(i)}]
\item $x_n\xrightarrow{\tau(M,M_\ast)} x\ \Leftrightarrow\ x_n\xrightarrow{ s(M,M_\ast)} x\ \Leftarrow\  x_n\xrightarrow{\tau_{o}(\saM)} x$.
\item When $M$ is $\sigma$-finite then  $x_n\xrightarrow{ s(M,M_\ast)} x\ \Leftrightarrow\  x_n\xrightarrow{\tau_{os}(\saM)} x$.
\end{enumerate}
\end{thm}
\begin{proof}
{\rm(i)}~First we show that $x_n\xrightarrow{\tau(M,M_\ast)} x\ \Leftrightarrow\ x_n\xrightarrow{ s(M,M_\ast)} x$. One direction follows from (\ref{e0}).  For the other direction note that if $ x_n\xrightarrow{ s(M,M_\ast)} x$ then $(x_n)_{n\in\N}$ is bounded.  Hence $x_n\xrightarrow{\tau(M,M_\ast)} x$ in view of (\ref{e1}).  The implication $x_n\xrightarrow{\tau_{o}(\saM)} x\ \Rightarrow\ x_n\xrightarrow{ s(M,M_\ast)} x$ follows from the fact that $\tau_{o}(\saM)$ is the finest topology that preserves  order convergence.

{\rm(ii)}~In virtue of Proposition \ref{1.1} it suffices to prove that  for every sequence $(x_n)_{n\in\N}$  converging  to $x$ w.r.t. $s(M,M_\ast)$ it is possible to extract a subsequence that order converges to $x$ in $(\saM,\le)$.  By the translation invariance of $\tau_o(\saM)$ we can suppose that $x=0$ and since $(x_n)_{n\in\N}$ is necessarily bounded we can further suppose that $(x_n)_{n\in\N}$ is a sequence in $\saM^1$.  The proof is based on a recursive application of Noncommutative Egoroff's Theorem \cite[Theorem 4.13, p. 85]{Takesaki}:  Let $(a_n)_{n\in\N}$ be a sequence  in a von Neumann algebra $M$ converging to $0$ w.r.t.  $s(M,M_\ast)$.  Then, for every projection $e$ in $M$, for every $\varphi\in M_\ast^+$ and for every $\varepsilon>0$, there exists a projection $e_0\le e$ and  a subsequence $(a_{n_k})_{k\in \N}$ such that $\varphi(e-e_0)<\varepsilon$ and $\Vert a_{n_k}e_0\Vert<2^{-k-1}$.

First we suppose that the sequence $(x_n)_{n\in\N}$ is positive.
Since $M$ is $\sigma$-finite, it  admits a faithful normal state $\psi$. Applying Egoroff's Theorem with $e:=\mathds 1$, $\varphi:=\psi$ and $\varepsilon=2^{-1}$, we obtain a projection $e_1$ and a subsequence $\Bigl(x^{(1)}_k\Bigr)_{k\in\N}$ of $(x_n)_{n\in\N}$ such that $\Bigl\Vert x^{(1)}_ke_1\Bigr\Vert<2^{-k-1}$ for each $k\in\N$ and $\psi(\mathds 1-e_1)<2^{-1}$.
The sequence $\Bigl(x^{(1)}_k\Bigr)_{k\in\N}$ converges  to $0$ w.r.t. $s(M,M_\ast)$ and so we can apply Egoroff's Theorem again for this sequence with $e:=\mathds 1-e_1$, $\varphi:=\psi$ and $\varepsilon= 2^{-2}$ to obtain a projection $e_2\le \mathds 1-e_1$, and a subsequence $\Bigl(x^{(2)}_k\Bigr)_{k\in\N}$ of $\Bigl(x^{(1)}_k\Bigr)_{k\in\N}$ such that $\Bigl\Vert x^{(2)}_ke_2\Bigr\Vert<2^{-k-2}$ and $\psi(\mathds 1-e_1-e_2)<2^{-2}$.  

An inductive application of Egoroff's Theorem yields a sequence of orthogonal projections $(e_n)_{n\in\N}$ satisfying $\psi(\mathds 1-\sum_{i=1}^n e_i)<2^{-n}$; and a nested sequence of subsequences $\Bigl(x^{(j)}_k\Bigr)_{k\in\N}$ of $(x_n)_{n\in\N}$ where $\Bigl(x^{(j+1)}_k\Bigr)_{k\in\N}$ is a subsequence of $\Bigl(x^{(j)}_k\Bigr)_{k\in\N}$ such that  $\Bigl\Vert x^{(j)}_ke_j\Bigr\Vert <2^{-k-j}$.

Let $p_n:=\sum_{i=1}^ne_i$.  Then $(\mathds 1-p_n)_{n\in\N}$ is a decreasing sequence of projections and $\psi(\wedge(\mathds 1-p_n))=0$ and thus, since $\psi$ is faithful  $1-p_n\downarrow 0$.

By the way the nested array $\Bigl(x^{(j)}_k\Bigr)_{k\in\N}$ is constructed, one can check that if $j\ge i$ then $x^{(j)}_j=x^{(i)}_p$ for some $p\ge j$.  Thus $\Bigl\Vert x^{(j)}_je_i\Bigr\Vert=\Bigl\Vert x^{(i)}_{p}e_i\Bigr\Vert<2^{-p-i}\le 2^{-j-i}$.

The sequence $\Bigl(x^{(j)}_j\Bigr)_{j\in\N}$ is a subsequence of $(x_n)_{n\in\N}$.  We claim that  $\Bigl(x^{(j)}_j\Bigr)_{j\in\N}$ order converges to $0$.  
To this end we observe that
\begin{align*}
 x^{(j)}_j  &
=  x^{(j)}_j p_j +p_j x^{(j)}_j(\mathds 1 -p_j)+(\mathds 1-p_j) x^{(j)}_j(\mathds 1-p_j)\\
&\le \bigl\| x^{(j)}_j p_j + p_j x^{(j)}_j(\mathds 1 -p_j)\bigr\|\mathds 1 + \mathds 1 -p_j\\
 & \le\bigl(1+ 2\| x^{(j)}_j p_j\|\bigr) \mathds 1  -p_j \\
&\le\biggl(1+2\sum_{i=1}^j\|x^{(j)}_j e_i\|\biggr)\mathds 1-p_j\\
& \le (1+2^{-j+1})\mathds 1 -p_j\,. 
\end{align*}

Since $(1+2^{-j+1})\mathds 1-p_j\downarrow 0$, it follows that $\Bigl( x^{(j)}_j\Bigr)_{j\in\N}$ order converges to $0$.

To complete the proof we consider the case when $(x_n)_{n\in\N}$ is not assumed to be in $M_+$.  If $x_n\st{M} 0$ then $|x_n|\st{M}0$ (where $|x_n|=\sqrt{x_n^2})$ and therefore $(|x_n|)_{n\in\N}$ converges to $0$ w.r.t. $\tau_{os}(\saM)$ by the above.  Thus $(|x_n|)_{n\in\N}$ has a subsequence $(|x_{n_k}|)_{k\in\N}$ that order converges to $0$, i.e. for which one can find a sequence $(y_k)_{k\in\N}$ in $M_+$ such that $0\le |x_{n_k}|\le y_k$ and $y_k\downarrow 0$.  The result then follows from $-y_k\le-|x_{n_k}|\le x_{n_k}\le |x_{n_k}|\le y_k$.
\end{proof}

\begin{cor}\label{3.2}
Assume that $M$ is $\sigma$-finite and $K$ is a bounded subset of $\saM$. Then the $s(M,M_\ast)$-closure of $K$ coincides with the $\tau_{os}(\saM)$-closure of $K$.
\end{cor}
\begin{proof}
When $M$ is $\sigma$-finite $s(M,M_\ast)|K$ is metrisable and therefore result follows from Theorem \ref{3.4}.
\end{proof}

\begin{cor}\label{3.3}
Assume that $M$ is $\sigma$-finite.  Then
\[\tau_s|K=s(M,M_\ast)|K=\tau(M,M_\ast)|K=\tau_o(\saM)|K=\tau_{os}(\saM)|K\]
for every bounded subset $K$ of $\saM$.
\end{cor}

Note that $s(M,M_\ast)|\saM$ and $\tau_o(\saM)$ are different unless $M$ is finite dimensional (see Remark \ref{r1}{\rm(i)}). The aim of the rest of this section is to  study when the  order topology $\tau_o(\saM)$ coincides with $\tau(M,M_\ast)|\saM$.

\begin{lem}\label{3.7}  $\gamma\bigl[\tau_u|\saM,s(M,M_\ast)|\saM\bigr]=\gamma\bigl[\tau_u,s^\ast(M,M_\ast)\bigr]|\saM$.
 \end{lem}
 \begin{proof}
 To simplify the notation let     $\gamma:=\gamma\bigl[\tau_u|\saM,s(M,M_\ast)|\saM\bigr]$.
  Since $s(M,M_\ast)|K=\gamma\bigl[\tau_u,s^\ast(M,M_\ast)\bigr]|K$  for every bounded subset $K$ of $\saM$ we get $\gamma\bigl[\tau_u,s^\ast(M,M_\ast)\bigr]|\saM\subseteq\gamma$ by Proposition \ref{W2.2} {\rm(ii)}.
      For the reverse inclusion let $\tilde M:=\saM\times \saM$.  Equipped with  addition and scalar multiplication on $\R$ defined  pointwise, $\tilde M$ is a vector space over $\R$ and the mapping $\tilde M\ni (x,y)\mapsto x+iy\in M$ is an isomorphism  of $\tilde M$ onto $M$ (as real vector spaces).  $(\tilde M,\gamma\times\gamma)$ is a Hausdorff topological vector space over $\R$. Denote by $\tau$ the Hausdorff real-linear topology induced on $M$ by $\gamma\times\gamma$. Then $\tau|K=s^\ast(M,M_\ast)|K$ for every bounded subset $K$ of $M$.  Hence $\tau\subseteq\gamma\bigl[\tau_u,s^\ast(M,M_\ast)\bigr]$ and therefore $\tau|\saM\subseteq \gamma\bigl[\tau_u,s^\ast(M,M_\ast)\bigr]|\saM$.  Finally, observe that $\tau|\saM=\gamma$ and hence the required inclusion holds.
\end{proof}

 Let $p$ be a nonzero projection in $M$.  We recall that $p$ is said to be a minimal projection if whenever $e$ is  a nonzero projection such that $0\ne e\le p$ then $e=p$.  Equivalently, $p$ is minimal if $pMp=\mathds C p$.  If $pMp$ is abelian then $p$ is said to be an abelian projection.  Every minimal projection is obviously abelian.  $M$ is said to be of type I if every nonzero central projection of $M$ majorizes  a nonzero abelian projection.  We recall that every type I factor is $\ast$-isomorphic to a $B(H)$ for some Hilbert space $H$.
A von Neumann algebra $M$ is said to be atomic if every nonzero projection majorizes a minimal projection.   Obviously if $M$ is atomic then $M$ is of type I.  Moreover, we say that $M$  \emph{purely atomic} if every von Neumann subalgebra of $M$ is atomic.  (When talking about subalgebra we do not require that subalgebra contains the unit of its superalgebra.) Observe that $M$ can be atomic without being purely atomic.  For example $B(H)$ when $H$ is infinite-dimensional and separable is atomic but not purely atomic because $L^\infty[0,1]$ can be identified with a von Neumann subalgebra of $B(H)$.

\begin{thm}\label{3.9}
The following three statements are equivalent:
\begin{enumerate}[{\rm(i)}]
\item The unit ball $M^1$ is $\sstar$-compact (and therefore on bounded parts of $M$ the $\sigma$-strong$^\ast$ topology coincides with the weak$^\ast$-topology),
\item $M$ is purely atomic,
\item $M$ is $\ast$-isomorphic to the direct sum of finite dimensional matrix algebras.
\end{enumerate}
\end{thm}
\begin{proof}
{\rm(i)}$\Rightarrow${\rm(ii)}.  Suppose that $M$ is not purely atomic and let $N$ be a von Neumann subalgebra of $M$ that is not atomic.  Without any loss of generality we can assume that $N$ has no minimal projections and that it is $\sigma$-finite.  Let $\varphi$ be a faithful normal state on $N$.  Using the noncommutative version of Lyapunov Theorem  \cite{AkemAnder} (or \cite{ChEnFu,Barbieri}) it is possible to define projections like the Rademacher functions: $\{p_{n,i}:n\in\N,\ i=1,\dots,2^n\}$ such that  $\mathds 1_N=p_{1,1}+p_{1,2}$ and $p_{n-1,i}=p_{n,2i-1}+p_{n,2i}$; and moreover $\varphi(p_{n,i})=2^{-n}$.  Put $e_n=\sum_{i=1}^{2^n}(-1)^ip_{n,i}$.  Then $(e_n)_{n\in\N}$ is a sequence of self-adjoint elements in the unit ball of $N$ and $\eta_\varphi(e_n-e_m)=2$ for every $n\neq m$.  This implies that the unit ball of $N$ is not $s^\ast(N,N_\ast)$-compact and thus result follows.

{\rm(ii)}$\Rightarrow${\rm(iii)}. Let $\centre{M}$ denote the centre of $M$.  If $z$ is a minimal projection of $\centre{M}$ then  $zM$ is a type I factor and therefore $zM$ is $\ast$-isomorphic to $B(H)$ for some Hilbert space $H$.  Note that $H$ cannot be infinite-dimensional because $M$ is purely atomic.    The result then follows by taking a family of pairwise orthogonal minimal  projections in $\centre{M}$ say $\{z_\lambda:\lambda\in \Lambda\}$ such that $\sum_{\lambda\in \Lambda}z_\lambda =\mathds 1$.

{\rm(iii)}$\Rightarrow${\rm(i)}. Let $M=\sum_{\lambda\in \Lambda}\oplus B(H_{n_\lambda})$ where $\Lambda$ is an indexing set and $n_\lambda\in\N$ for every $\lambda\in\N$. Denote by $\Vert\cdot\Vert_\lambda$ the norm on $B(H_{n_\lambda})$ and by $B_\lambda$ its unit ball.  Then $(B_\lambda,\|\cdot\|_\lambda)$ is compact  for every $\lambda \in \Lambda$ and therefore the  product space $\Pi_{\lambda\in \Lambda}(B_\lambda,\|\cdot\|_\lambda)$ is compact by Tychonoff Theorem.  Observe that $(M^1,\sstar|M^1)$ is homeomorphic with $\Pi_{\lambda\in \Lambda}(B_\lambda,\|\cdot\|_\lambda)$ and therefore result follows. 
\end{proof}

\begin{thm}\label{3.10}
The following statements are equivalent:
\begin{enumerate}[{\rm(i)}]
\item  $\tau(M,M_\ast)|\saM=\tau_{os}(\saM)$, 
\item  $\tau(M,M_\ast)|\saM$ is sequential,
\item $\tau(M,M_\ast)$ is sequential,
\item $M$ is $\sigma$-finite and $\tau_{o}(\saM)$ is a linear topology,
\item $M$ is $\sigma$-finite and satisfies one (and therefore all) of the equivalent conditions of Theorem \ref{3.9}.
\end{enumerate}
\end{thm}
\begin{proof}
{\rm(i)}$\Rightarrow${\rm(ii)}.  If $\tau_{os}(\saM)=\tau(M,M_\ast)|\saM$ then $\tau(M,M_\ast)|\saM$ is sequential since $\tau_{os}(\saM)$ is obviously sequential.  

{\rm(ii)}$\Rightarrow${\rm(v)}. Suppose that $\tau(M,M_\ast)|\saM$ is sequential.  Observe that $M$ must be $\sigma$-finite because otherwise it contains an uncountable family $\{p_\gamma:\gamma\in\Gamma\}$ of  nonzero orthogonal projections and then the set 
\[N:=\biggl\{x\in M_{sa}:\exists \Gamma_0\subseteq\Gamma,\, |\Gamma_0|\le\aleph_0,\, 0\le x\le\bigvee_{\gamma\in\Gamma_0}p_\gamma\biggr\}\]
 is sequentially $\tau(M,M_\ast)$-closed but not $\tau(M,M_\ast)$-closed.   We recall that   $s(M,M_\ast)|\saM^1$ is metrizable when $M$ is $\sigma$-finite.  Let us show that $M^1$ is $s^\ast(M,M_\ast)$-compact.   Since $M^1\subseteq \saM^1+i\saM^1$ is  $\sstar$-closed, it suffices to show that $\saM^1$ is $s(M,M_\ast)$-compact.  If $\saM^1$ is not $s(M,M_\ast)$-compact then we can apply Theorem \ref{2.5} with $X:=\saM$, $\tau':=s(M,M_\ast)|\saM$ and $\tau:=\tau_u|\saM$ to deduce that the mixed topology $\gamma\bigl[\tau_u|\saM,s(M,M_\ast)|\saM]$ is not sequential and therefore {\rm(v)} follows by Lemma \ref{3.7} and Theorem \ref{mixed}.
 
 {\rm(v)}$\Rightarrow${\rm(iii)}.  Assume that $M$ is $\sigma$-finite and that it is $\ast$-isomorphic to the direct sum of finite dimensional matrix algebras, say $M=\sum_{\lambda\in \Lambda}\oplus B(H_{n_\lambda})$ where $\Lambda$ is countable.  We can apply Corollary \ref{2.8}  with $X:=M$ and $(\rho_\lambda)_{\lambda\in\Lambda}$ defined by $\rho_\lambda(x):=\|x_\lambda\|_\lambda$ where $x=(x_\lambda)_{\lambda\in\Lambda}$ and $\|\cdot\|_\lambda$ denotes the norm on $B(H_{n_\lambda})$ to deduce that $\gamma[\tau,\tau']$ is sequential.  Obviously $\tau$ coincides with $\tau_u$ and it is easily seen that on bounded parts of $M$ the topology $\tau'$ (=product topology) coincides with $s^\ast(M,M_\ast)$.  Hence (see comment following Proposition \ref{W2.2}) $\gamma[\tau,\tau']=\gamma\bigl[\tau_u,s^\ast(M,M_\ast)\bigr]$.  Thus {\rm(iii)} follows by Theorem \ref{mixed}.
 
{\rm(iii)}$\Rightarrow${\rm(i)}.  If $\tau(M,M_\ast)$ is sequential then $\tau(M,M_\ast)|\saM$ is sequential and therefore $M$ is $\sigma$-finite.   Therefore we get $\tau(M,M_\ast)|\saM=\tau_{os}(\saM)$ in virtue of Theorem \ref{3.4}.

{\rm(i)}$\Leftrightarrow${\rm(iv)}.  The implication {\rm(i)}$\Rightarrow${\rm(iv)} is trivial.  In virtue of  Lemma \ref{3.7} and Theorem \ref{mixed} we have  
 $\tau(M,M_\ast)|\saM=\gamma\bigl[\tau_u,s^\ast(M,M_\ast)\bigr]|\saM=\gamma\bigl[\tau_u|\saM,s(M,M_\ast)|\saM\bigr]$, i.e. $\tau(M,M_\ast)|\saM$ is the finest linear topology on $\saM$ that agrees with $s(M,M_\ast)|\saM$ on bounded subsets of $\saM$.  Thus in view of Corollary \ref{3.3} we get $\tau_{os}(\saM)\subseteq\tau(M,M_\ast)|\saM$ when $\tau_o(\saM)$ is linear.
 \end{proof}

\section{The order topology and the sequential order topology on $\saM^1$ and $P(M)$}\label{section4}

The order topology on the projection lattice of a Hilbert space was studied in \cite{Palko} and \cite{BCW}.  Let $\mathcal L$ denote the lattice of projections on a separable Hilbert space $H$.  Using \eqref{e2} and {\rm ({2.}\ref{eq3})} we immediately get $\tau_s|\mathcal L\subseteq \tau_o(\mathcal L)$.  \cite[Example 2.4]{Palko} shows that if $\dim H\ge 2$ then $\tau_o(\mathcal L)\nsubseteq \tau_u|\mathcal L$ and therefore $\tau_s|\mathcal L\neq\tau_o(\mathcal L)$.  (This is in contrast with Corollary \ref{3.3}.)  In relation to this we mention that in \cite[Theorem 20]{BCW}  the authors show that when $B$ is maximal Boolean sublattice of $\mathcal L$ then $\tau_o(\mathcal L)|B=\tau_s|B$.  Let us point out that in fact this follows from Proposition \ref{p1} and from Corollary \ref{3.3}.  Indeed, if $B$ is a maximal Boolean sublattice of $\mathcal L$ then $B$ is the projection lattice of a maximal abelian $\ast$-subalgebra $M$ of $B(H)$ and therefore  
 \[\tau_s|B=s(M,M_\ast)|B=\tau_o(\saM)|B=\tau_o(B).\] 
 
 When $\dim H<\infty$ the order topology on $\mathcal L$ coincides with the discrete topology and therefore it is finer than the restriction of the uniform topology but \cite[Example 2.3]{Palko} shows that if $\dim H=\infty$ then $\tau_u|\mathcal L\nsubseteq \tau_o(\mathcal L)$.    In \cite{Palko} V.~Palko conjectured that $\tau_s|\mathcal L=\tau_o(\mathcal L)\,\cap\,\tau_u|\mathcal L$.  This is obviously true when $\dim H<\infty$ and in full agreement with the conjecture, he proved that a sequence of atoms in $\mathcal L$ converges to $0$ w.r.t. $\tau_s$ if and only if it converges to $0$ w.r.t. $\tau_u|\mathcal L\,\cap\,\tau_o(\mathcal L)$.  \cite[Example 16]{BCW} however shows that  $\tau_s|\mathcal L=\tau_o(\mathcal L)\,\cap\,\tau_u|\mathcal L$ holds only when $\dim H<\infty$.

In this section we study the order topology and the sequential order topology on $\saM^1$ and $P(M)$.    Proposition \ref{1.7} already gives that $\tau_o(\saM)|\saM^1\subseteq\tau_o(\saM^1)$ and  $\tau_{o}(\saM)|P(M)\subseteq \tau_{o}(P(M))$.  (Similar inclusions hold for the sequential order topology in view of the comment following Proposition \ref{1.6}.)  

We shall now exhibit an example that will be used later.  It is in fact a construction given in  \cite[Example 26]{BCW}.   

\begin{exmp}\label{4.1}
Let $(\xi_n)_{n\in\N}$ be an orthonormal basis of a separable Hilbert space $H$.
For each $n$  let $p_n$   denote the projections of $H$ onto $\overline{\mbox{span}}\{\frac 1n\xi_1+\xi_n, \xi_{n+1},\xi_{n+2},\ldots\}$.
 Let $M=B(H)$.  Then
\begin{enumerate}[{\rm(i)}]
\item $(p_n)_{n\in\N}$  converges to $0$ w.r.t.  $\tau_s$,
\item $(p_n)_{n\in\N}$  converges to $0$ w.r.t.  $\tau_{os}(M_{sa)}$,
\item $(p_n)_{n\in \N}$  does not converge to $0$ w.r.t $\tau_{os}(M_{sa}^1)$,
\item $(p_n)_{n\in\N}$  does not converge to $0$ w.r.t. $\tau_{os}(P(M))$.
\end{enumerate}
\end{exmp}
\begin{proof}
 {\rm (i)} and {\rm(iv)} were proved in \cite[Example 26]{BCW}.  {\rm(ii)} follows from Theorem \ref{3.4} and \eqref{e1}.
To prove {\rm(iii)} suppose that $(a_{n_k})_{k\in\N}$ is a sequence in $\saM^1$ such that $p_{n_k}\le a_{n_k}$ for every $k\in\N$.  Then {\rm(ii)} of Lemma \ref{4.2} yields $a_{n_k}\ge \bigvee_{i\ge k}p_{n_i}\ge p$ for every $k\in\N$ where $p$ denotes the projection of $H$ onto the one-dimensional subspace spanned by $\xi_1$.  
\end{proof}

We recall that two projections $e$ and $f$ in $M$ are said to be equivalent (in symbols $e\sim f$) if there exists $u\in M$ such that $uu^\ast=e$ and $u^\ast u=f$.   A projection $e$ is said to be finite if whenever $f$ is a projection such that $e\sim f$ and $f\le e$ then $e=f$.  If $e$ is not finite then it is infinite.  Moreover, $e$ is said to be properly infinite if $ze$ is infinite or $0$ for every $z\in \mathcal Z(M)$.  $M$ is said to be finite, infinite or properly infinite according to the property of the identity projection $\mathds 1$. Moreover, there are two orthogonal projections $z_f$ and $z_i$ in $\centre{M}$ such that $z_f$ is finite, $z_i$ is properly infinite and $z_f+z_i=\mathds 1$.  We further recall that if $M$ is properly infinite  then  there is a sequence of mutually equivalent and pairwise orthogonal projections $(e_n)_{n\in\N}$ such that $\bigvee_{n\in\N}e_n=\mathds 1$.  These projections, together with the  partial isometries implementing their equivalence, generate a Type I subfactor (i.e. a unital von Neumann subalgebra that is a factor) $N$ of $M$  that is $\ast$-isomorphic to $B(H)$ for some separable infinite-dimensional Hilbert space $H$.  Thus, by Example \ref{4.1}, and since $P(N)$ is a complete sublattice of $P(M)$ it follows that a properly infinite von Neumann algebra $M$ contains a sequence of projections which is $\sigma$-strongly null but not $\tau_o(P(M))$-null.  This observation is in part a motivation for Theorem \ref{4.4} in which we give a new characterization of finite von Neumann algebras.

We further recall that in the proof of Proposition \ref{p1} we have seen  that when $N=B(H_2)$ then $\tau_o(N_{sa}^1)|P(N)=\tau_o(P(N))$, i.e. unlike {\rm ({2.}\ref{eq1})} and {\rm ({2.}\ref{eq3})}, in {\rm ({2.}\ref{eq2})} we can have an equality without the algebra being abelian.

\begin{lem}\label{4.3}
Let $(p_i)$ be a decreasing sequence  of projections in $M$.  Then the sequence $\bigl(2^{-i}(\mathds 1-p_i)+p_i\bigr)$ is decreasing.
\end{lem}
\begin{proof}
Indeed,
\begin{align*}
2^{-i-1}(\mathds 1-p_{i+1})+p_{i+1}&
=2^{-1}(2^{-i}(\mathds 1-p_i))+p_i-(\mathds 1-2^{-i-1})(p_i-p_{i+1})\\
&\le 2^{-i}(\mathds 1-p_i)+p_i.
\end{align*}
\end{proof}

  Let us recall that if $p$ and $q$ are projections in $M$ then $p\vee q -p \sim q - p\wedge q$ \cite[p. 292,  Proposition V.1.6]{Takesaki}.  Hence, if a state $\psi$ on $M$ is tracial (i.e. $\psi(x^\ast x)=\psi(xx^\ast)$ for all $x\in M$), then its restriction to $P(M)$ is a valuation, i.e. $\psi(p\vee q)+\psi(p\wedge q)=\psi(p)+\psi(q)$ for any $p,q\in P(M)$.  Consequently, $\psi$ is subadditive, i.e. $\psi(p\vee q)\le\psi(p)+\psi(q)$ for any $p,q\in P(M)$; if, moreover,  $\psi$ is normal, then it is even $\sigma$-subadditive, i.e. $\psi\bigl(\bigvee_n p_n\bigr)\le \sum_n \psi(p_n)$ for every sequence $(p_n)_{n\in\N}$ in $P(M)$.  (In \cite{BunHa} it is shown that, conversely, every subadditive probability measure on $P(M)$ arises in this way.)

\begin{thm}\label{4.4}
Let $M$ be a $\sigma$-finite von Neumann algebra.  Then the following statements are equivalent.
\begin{enumerate}[{\rm(i)}]
\item $M$ is finite.
\item Every sequence $(x_n)_{n\in\N}$ in $M_{sa}^1$ converging $\sigma$-strongly to $0$ converges to $0$ w.r.t. $\tau_{os}(M_{sa}^1)$.
\item If $(p_n)_{n\in\N}$ is a sequence in $P(M)$ converging $\sigma$-strongly to $0$, then there exists a subsequence $(p_{n_i})$ such that
\[\limsup_{i}p_{n_i}=0\,.\]
\item $\tau_o(M_{sa}^1)|P(M)=\tau_o(P(M))$
\end{enumerate}
\end{thm}

\begin{proof}
{\rm (i)}$\Rightarrow${\rm(ii)}.  Let $(x_n)_{n\in\N}$ be a sequence in $M_{sa}^1$ convergent $\sigma$-strongly to $0$.  We shall first suppose that $x_n\ge 0$ for each $n$.  We need to exhibit a subsequence $(x_{n_i})_{i\in\N}$ of $(x_n)_{n\in\N}$ that order converges to $0$.  Since $M$ is finite and $\sigma$-finite, $M$ admits a faithful, normal, tracial state $\psi$.  Since $\psi(x_n^2)=\rho_\psi(x_n)^2\rightarrow 0$, we can extract a subsequence $(x_{n_i})_{i\in\N}$ of $(x_n)_{n\in\N}$ such that $\psi(x_{n_i})\le\sqrt{\psi(x_{n_i}^2)}<4^{-i}$.  For each $i\in\N$ and $\lambda\in\R$, let $e^i(\lambda)$ be the projection in $M$ corresponding to the characteristic function associated with $\mbox{sp}(x_{n_i})\cap (-\infty,\lambda]$ -- i.e. $\{e^i(\lambda)\}_{\lambda\in\R}$ is the spectral resolution of $x_{n_i}$.  Then
\begin{align*}
0\le x_{n_i}&\le 2^{-i}\,e^i(2^{-i})+(\mathds 1-e^i(2^{-i}))\\
&=2^{-i}\biggl(\bigwedge_{j\ge i}e^j(2^{-j})+e^i(2^{-i})-\bigwedge_{j\ge i}e^j(2^{-j})\biggr)+(\mathds 1-e^i(2^{-i}))\\
&\le 2^{-i}\bigwedge_{j\ge i} e^j(2^{-j})+e^i(2^{-i})-\bigwedge_{j\ge i} e^j(2^{-j})+\mathds 1-e^i(2^{-i})\\
&=2^{-i}\bigwedge_{j\ge i} e^j(2^{-j})+\bigvee_{j\ge i}\bigl(\mathds 1-e^{j}(2^{-j})\bigr).
\end{align*}
Let
\[y_i=2^{-i}\bigwedge_{j\ge i} e^j(2^{-j})+\bigvee_{j\ge i}\bigl(\mathds 1-e^{j}(2^{-j})\bigr).\]
Then $0\le x_{n_i}\le y_i\le\mathds 1$ and by Lemma \ref{4.3} we know that the sequence $(y_i)_{i\in\N}$ is decreasing.  Thus, $\bigwedge_{i\in\N}y_i$ exists in $M_{sa}$ and $\bigwedge_{i\in\N}y_i\ge 0$.  The normality of $\psi$ entails that $\psi\bigl(\bigwedge_{i\in\N}y_i\bigr)=\lim_{i\to\infty}\psi(y_i)$. Since $2^{-j}\bigl(\mathds 1-e^j(2^{-j})\bigr)\le x_{n_j}$, it follows that $\psi\bigl(\mathds 1-e^j(2^{-j})\bigr)\le 2^j\psi(x_{n_j})<2^{-j}$.

Since $\psi$ is $\sigma$-subadditive we can estimate:
\[\psi(y_i)\le 2^{-i}+\sum_{j\ge i}\psi\bigl(\mathds 1-e^j(2^{-j}\bigr)<2^{-i}+\sum_{j\ge i}2^{-j}=3\,(2^{-i}).\]
Thus, $\psi\bigl(\bigwedge_{i\in\N}y_i\bigr)=0$ and therefore, since $\psi$ is faithful, it follows that $\bigwedge_{i\in\N}y_i=0$, i.e. $(x_{n_i})_{i\in\N}$ is order convergent to $0$.

If not every element $x_n$  is  positive, then we can   consider the sequence $(|x_n|)_{n\in\N}$ which is again $\sigma$-strongly convergent to $0$.    Then $(|x_n|)_{n\in\N}$ has a subsequence $(|x_{n_i}|)_{i\in\N}$ that order converges to $0$, i.e. for which one can find a sequence $(y_i)_{i\in\N}$ in $M_+^1$ such that $0\le |x_{n_i}|\le y_i$ and $y_i\downarrow 0$.  The result then follows from the following inequality.
\[-\mathds 1\le-y_i\le-|x_{n_i}|\le x_{n_i}\le |x_{n_i}|\le y_i\le \mathds 1\qquad (\text{for all }i\in\N).\]

{\rm (ii)}$\Rightarrow${\rm(iii)}.  If $(p_n)_{n\in\N}$ is a sequence of projections that converges $\sigma$-strongly to $0$ then $p_n\xrightarrow{\tau_{os}(\saM^1)} 0$ by (ii). Therefore, by Proposition \ref{1.1}, $(p_n)_{n\in\N}$ has a subsequence $(p_{n_i})_{i\in\N}$  order converging  to $0$ in $(M^1_{sa},\le)$. Thus there is a sequence $(y_i)_{i\in\N}$ in $M^1_{sa}$ such that $0\le p_{n_i}\le y_i$ and $y_i\downarrow 0$.  From \lemref{4.2} we obtain that $\bigvee_{j\ge i}p_{n_j}\le y_i$ for every $i\in\N$
and therefore $0\le \bigwedge_{i\in\N}\bigvee_{j\ge i}p_{n_j}\le\bigwedge_{i\in\N}y_i=0$, i.e. $\limsup_i p_{n_i}=0$.

{\rm (iii)}$\Rightarrow${\rm(i)}.  This was shown in the paragraph before Lemma \ref{4.3}.

{\rm (iv)}$\Rightarrow${\rm(i)}.
If $N$ is a $W^\ast$-subalgebra of $M$ then $N_{sa}^1$ and $P(N)$ are $s(M,M_\ast)$-closed. The hypothesis together with Propositions \ref{1.6} and \ref{1.7} imply that 
\begin{multline*}
\tau_o(N_{sa}^1)|P(N)\supseteq \tau_o(\saM^1)|P(N)=\tau_o(P(M))|P(N)\\=\tau_o(P(N))\supseteq\tau_o(N_{sa}^1)|P(N).
\end{multline*}
Hence, in view of the discussion in the paragraph before Lemma \ref{4.3} it is enough to show that $\tau_o(M_{sa}^1)|P(M)\neq\tau_o(P(M))$ when $M=B(H)$ for  a separable infinite-dimensional Hilbert space $H$.

Let $(\theta_n)_{n\in\N}$ be a sequence in $(\pi/4,\pi/2)$ such that $\theta_n\uparrow \pi/2$, and let $\sigma_n:=\sin\theta_n$ and $\gamma_n:=\cos\theta_n$.  Fix an orthonormal basis $(\xi_n)_{n\in\N}$ of $H$ and define $\eta_n:=\sigma_n\xi_1+\gamma_n\xi_n$.  Denote by $e_n$ the projection of $H$ onto $\mbox{span}\{\xi_n\}$; $q_n$ the projection of $H$ onto $\mbox{span}\{\eta_n\}$ and $f_n$ the projection of $H$ onto $\overline{\mbox{span}}\{\xi_{n},\xi_{n+1}\dots\}$

We show that $\sigma_n^4 e_1-e_n\le q_n$ for every $n\in\N$.  It is enough to consider vectors in the two-dimensional subspace spanned by $\xi_1$ and $\xi_n$.  Thus we we can express $e_1, e_n, q_n$ in matrix form (relative to the vectors $\xi_1$ and $\xi_n$).  
Writing $q_n-\sigma_n^4 e_1+ e_n$ in matrix form:
\[
\begin{pmatrix}
\sigma_n^2-\sigma_n^4 &\gamma_n \sigma_n\\
\gamma_n \sigma_n &\gamma_n^2+1
\end{pmatrix}=
\begin{pmatrix}
\gamma_n^2 \sigma_n^2 & \gamma_n \sigma_n\\
\gamma_n \sigma_n & \gamma_n^2+1
\end{pmatrix}
\]
one sees that  $\sigma_n^4 e_1-e_n\le q_n$.  Thus 
\[x_n:=\sigma_n^4 e_1-f_n\le q_n+f_{n+1}\le e_1+f_n=:y_n\,,\]
the sequences $(x_n)_{n\in\N}$ and $(y_n)_{n\in\N}$ are in $\saM^1$, $x_n\uparrow e_1$ and $y_n\downarrow e_1$.  Thus $p_n:=q_n+f_{n+1}\to e_1$ w.r.t. $\tau_o(\saM^1)$.  
 On the other-hand $\bigwedge_{i\ge n} p_i =0$.

{\rm (i)}$\Rightarrow${\rm(iv)}.
As observed before,  $M$ admits a faithful, normal, tracial state $\psi$. Then $\psi|(P(M)$ is a valuation. Thus $d(p,q):=\psi(p\vee q)-\psi(p\wedge q)$ defines by \cite[p. 230, Theorem X.1.1]{Birkhoff2} a metric on $P(M)$. We first prove the following estimation:
If $x,y\in M_{sa}$ and $p,q\in P(M)$ with $x\leq p\leq y\leq 1$ and  $x\leq q\leq y$, then
$$d(p,q)\leq 2\psi(y-x).$$
In fact, $d(p,q)=(\psi(p\vee q)-\psi(p))+(\psi(p)-\psi(p\wedge q))=(\psi(p\vee q)-\psi(p))+(\psi(p\vee q)-\psi(q))\leq 2(\psi(y)-\psi(x))=2\psi(y-x).$

In the last inequality we used that by Lemma \ref{4.2} we have $p\vee q\leq y$. 

In view of {\rm(2.\ref{eq2})} and Proposition \ref{1.1}, for the proof of (iv) it suffices to show that any sequence $(p_n)_{n\in\N}$ in $P(M)$ order converging in $(M^1_{sa},\leq)$ to $p\in P(M)$ has a subsequence order converging in $(P(M),\leq)$. Let now $p_n,p\in P(M)$ and $x_n, y_n\in M^1_{sa}$ such that $x_n\uparrow p$, $y_n\downarrow p$ and $x_n\leq p_n\leq y_n$. Then $y_n-x_n\downarrow 0$ and $d(p_n,p)\leq 2\psi(y_n-x_n)\rightarrow 0$. Therefore $(p_n)_{n\in\N}$ converges to $p$ in the metric lattice $(P(M),d)$. It follows from the proof of \cite[p. 246, Theorem X.10.16]{Birkhoff2} that $(p_n)_{n\in\N}$ has a subsequence  order converging to $p$ in $(P(M),\leq)$.
\end{proof}


We remark that Theorem \ref{4.4} does not imply  that for finite, $\sigma$-finite algebras the restriction of $s(M,M_\ast)$ to $\saM^1$ (resp. $P(M)$) coincides with the order topology $\tau_o(\saM^1)$ (resp. $\tau_o(P(M))$) -- see Proposition \ref{p1}.  
We also note that in the proof of the implications {\rm(ii)}$\Rightarrow${\rm(iii)}, {\rm(iii)}$\Rightarrow${\rm(i)} and {\rm(iv)$\Rightarrow$\rm(i)}  the assumption that $M$ is $\sigma$-finite is not used. \\

{\bf Acknowledgement.}
This work was supported by  the ``Grant Agency of the Czech Republic"
  grant number P201/12/0290,  ``Topological and geometrical properties of Banach spaces and operator algebras".


\begin{thebibliography}{99}

\bibitem{Aarnes} J.F. Aarnes, On the Mackey topology for a von Neumann algebra, Math. Scand. {\bf 22} (1968), 87--107.
\bibitem{AliprantisBurkinshaw} C.~Aliprantis, O.~Burkinshaw, Locally solid Riesz spaces,
Pure and Applied Mathematics, Vol. 76. Academic Press [Harcourt Brace Jovanovich, Publishers], New York-London, 1978.
\bibitem{Akemann} C.~A.~Akemann, The dual space of an operator algebra, Trans. Amer. Math. Soc. {\bf 126} (1967), 286–-302.
\bibitem{AkemAnder} C.~A.~Akemann, J.~Anderson, Lyapunov theorems for operator algebras, Mem. Amer. Math. Soc. {\bf 94}, (1991)  no. 458, iv+88 pp.
\bibitem{Barbieri} G.~Barbieri, Lyapunov's theorem for measures on {D}-posets,
  Internat. J. Theoret. Phys. \textbf{43} (2004), 1613--1623. 
\bibitem{Birkhoff1} G.~Birkhoff, On the structure of abstract algebras, Proc. Cambridge Philos. Soc. {\bf 31} (1935), 433--454.
\bibitem{Birkhoff2} G.~Birkhoff, Lattice theory, Amer. Math. Soc. Colloq. Publ. {\bf 25} Amer. Math. Soc. Providence, R.I., 1995.
\bibitem{Blackadar}B.~Blackadar, Operator algebras, Encyclopaedia of Mathematical Sciences,
  vol. 122, Springer-Verlag, Berlin, 2006, Theory of $C{^{*}}$-algebras and von
  Neumann algebras, Operator Algebras and Non-commutative Geometry, III.
  \bibitem{BCW} D.~Buhagiar, E.~Chetcuti, H.~Weber, The order topology on the projection lattice of a Hilbert space, Topology  Appl. {\bf 159},
(2012), 280--289.
\bibitem{BunHa} L.~J.~Bunce and J.~Hamhalter, Traces and subadditive measures on projections in JBW-algebras and von Neumann
 algebras, Proc. Amer. Math. Soc., {\bf 123} (1995),  157--160.
\bibitem{ChEnFu} H.~Choda, M.~Enemoto, M.~Fuji,  Liapounoff's theorem, Math. Japon. {\bf 28} (1983), 651–-653.
\bibitem{Floyd} E.~E.~Floyd, Boolean algebras with pathological order topologies, Pacific J. Math. {\textbf 5} (1955), 687–-689.
\bibitem{FloydKlee} E.~E.~Floyd, V.~L.~Klee, A characterization of reflexivity by
the lattice of closed subspaces, Proc. Amer. Math. Soc. {\bf 5} (1953), 655--661.
\bibitem{Fremlin} D.~H.~Fremlin,  Measure theory Vol. 2, Broad foundations, Torres Fremlin, Colchester, 2003. 
\bibitem{Frink} O.~Frink Jr., Topology in lattices, Trans. Amer. Math. Soc. {\bf 51} (1942), 569–-582.
\bibitem{Halmos} P.~R.~Halmos, A Hilbert Space Problem Book, second ed., Springer Verlag, New York, Berlin, 1982.
\bibitem{Kantorovich} L.~V.~Kantorovich, O poluuporjado\v cennych linejnych postranstvach i ich primeneniach k teorii linejnych operacij, Dokl. Akad. Nauk. {\bf 4} (1935), 11--14.
\bibitem{LuxemburgZaanen}W.~A.~J. Luxemburg and A.~C. Zaanen, Riesz spaces. {V}ol. {I},
  North-Holland Publishing Co., Amsterdam, 1971, North-Holland Mathematical
  Library.
\bibitem{MaAn} J.~C.~Mathews, R.~F.~Anderson, A comparison of two modes of order convergence, Proc. Amer. Math. Soc. {\bf 18} (1967), 100--104.
\bibitem{McShane} E.~McShane, Order preserving maps and integration processes, Ann. of Math. {\bf 31} Princeton 1953.
\bibitem{Novak} M.~Nowak, On the finest Lebesgue topology on the space of essentially bounded measurable functions, Pacific J. Math. {\bf 140} (1989), 155--161.
\bibitem{Palko} V.~Palko, The weak convergence of unit vectors to zero in the Hilbert space is the convergence of one-dimensional subspaces in the order topology,  Proc. Amer. Math. Soc. {\bf 123} (1995), 715–-721.
\bibitem{Sakai} S.~Sakai, $C\sp*$-algebras and $W\sp*$-algebras, Ergebnisse der Mathematik und ihrer Grenzgebiete, Band 60, Springer-Verlag, New York-Heidelberg, 1971.
\bibitem{Schaeffer} H.~H.~Schaefer,  Topological vector spaces,  Springer Verlag, 1971. 
\bibitem{Sherman} S.~Sherman, Order in operator algebras,
Amer. J. Math. {\bf 73} (1951),  227--232.
\bibitem{Takesaki} M.~Takesaki, Theory of Operator Algebras I, Springer Verlag, 1979.
\bibitem{Vladimirov} D.~A.~Vladimirov, Boolesche Algebren, Akademie-Verlag Berlin, 1978.
\bibitem{Wiweger} A.~Wiweger, Linear spaces with mixed topologies, Studia Math. {\bf 20} (1961), 47--68.

\end{thebibliography}
\end{document}